\documentclass[final]{amsart}
\usepackage{amssymb}
\usepackage{array}
\usepackage{graphicx}
\usepackage{hhline}

\newtheorem{theorem}{Theorem}[section]
\newtheorem{lemma}[theorem]{Lemma}
\newtheorem{proposition}[theorem]{Proposition}
\newtheorem{corollary}[theorem]{Corollary}
\newtheorem{claim}[theorem]{Claim}

\theoremstyle{definition}
\newtheorem{definition}[theorem]{Definition}

\newtheorem{conjecture}[theorem]{Conjecture}

\newcommand{\Aut}{\mathrm{Aut\mkern 2mu}}
\newcommand{\M}{\mathrm{M\mkern 1mu}}

\newcommand{\A}{\mathcal A}

\newcommand{\IN}{\mathbb N}

\title[Automorphism groups of  superextensions of finite monogenic semigroups]{Automorphism groups of  superextensions\\ of finite monogenic semigroups}
\author{Taras Banakh and Volodymyr Gavrylkiv}
\address[T.~Banakh]{Ivan Franko National University of Lviv (Ukraine) and Jan Kochanowski University in Kielce (Poland)} \email{t.o.banakh@gmail.com}
\address[V.~Gavrylkiv]{Vasyl Stefanyk Precarpathian National University,
Ivano-Frankivsk, Ukraine} \email{vgavrylkiv@gmail.com}

\begin{document}

\begin{abstract}
A family $\mathcal L$ of
subsets  of a set $X$ is called {\em linked} if $A\cap
B\ne\emptyset$ for any $A,B\in\mathcal L$.  A linked family
$\mathcal M$ of subsets of $X$ is  {\em maximal linked} if
$\mathcal M$ coincides with each linked family $\mathcal L$ on
$X$ that contains $\mathcal M$. The {\em  superextension}
$\lambda(X)$ of $X$ consists of all maximal linked families on $X$. Any
associative binary operation $* : X\times X \to X$ can be extended
to an associative binary operation $*:
\lambda(X)\times\lambda(X)\to\lambda(X)$. In the
paper we study automorphisms of the superextensions of  finite
monogenic semigroups and characteristic ideals in such semigroups. In particular, we describe the automorphism groups of the superextensions of
finite monogenic semigroups of cardinality $\leq 5$.
\end{abstract}
\maketitle

\section{Introduction}

In this paper we investigate the automorphism group of the
superextension $\lambda(S)$ of a  finite monogenic semigroup $S$. The
thorough study of various extensions of semigroups was started in
\cite{G2} and continued in \cite{BG2}--\cite{BGN} and \cite{G4}--\cite{G8}.
The largest among these extensions is the semigroup $\upsilon(S)$ of all
upfamilies on $S$. A family $\mathcal{A}$ of non-empty subsets of
a set $X$ is called an {\em upfamily} if for each set
$A\in\mathcal{A}$ any subset $B\supset A$ of $X$ belongs to
$\mathcal{A}$. Each family $\mathcal{B}$ of non-empty subsets of
$X$ generates the upfamily $\langle B \subset X :
B\in\mathcal{B}\rangle = \{ A \subset X : \exists B \in\mathcal{B}\
( B \subset A )\}$. An upfamily $\mathcal{F}$  that is closed
under taking  finite intersections is called a {\em filter}. A
filter $\mathcal{U}$ is called an {\em ultrafilter} if
$\mathcal{U} = \mathcal{F}$ for any filter $\mathcal{F}$
containing $\mathcal{U}$. The family $\beta(X)$ of all
ultrafilters on a set $X$ is called {\em the Stone-\v Cech
extension of $X$}, see \cite{HS}. An ultrafilter, generated by a singleton $\{ x \}$, $x\in X$, is called
{\em principal}. Each point $x\in X$ is identified with the
principal ultrafilter $\langle\{ x \}\rangle$ generated by the
singleton $\{ x \}$, and hence we can consider
$X\subset\beta(X)\subset\upsilon(X)$. It was shown in \cite{G2}
that any associative binary operation $* : S\times S \to S$ can be
extended to an associative binary operation $*:
\upsilon(S)\times\upsilon(S)\to\upsilon(S)$ by the formula
\begin{equation*}
\mathcal A*\mathcal B=\big\langle\bigcup_{a\in A}a*B_a:A\in\mathcal A,\;\;\{B_a\}_{a\in A}\subset\mathcal B\big\rangle
\end{equation*}
for upfamilies $\mathcal{A}, \mathcal{B}\in\upsilon(S)$. In this
case the Stone-\v Cech compactification $\beta(S)$ is a
subsemigroup of the semigroup $\upsilon(S)$.

The semigroup $\upsilon(S)$ contains as subsemigroups many other important
extensions of $S$. In particular, it contains the semigroup
$\lambda(S)$ of maximal linked upfamilies, see \cite{G1}, \cite{G2}.
An upfamily $\mathcal L$ of
subsets  of $S$ is said to be {\em linked} if $A\cap
B\ne\emptyset$ for all $A,B\in\mathcal L$.  A linked upfamily
$\mathcal M$ of subsets of $S$ is  {\em maximal linked} if
$\mathcal M$ coincides with each linked upfamily $\mathcal L$ on
$S$ that contains $\mathcal M$. It follows that $\beta(S)$ is a subsemigroup of $\lambda(S)$.
The space $\lambda(S)$
is well-known in  General and Categorial Topology as the {\em
superextension} of $S$, see \cite{vM}--\cite{Ve}.

For a finite set $X$, the cardinality of the set $\lambda(X)$ grows very quickly as $|X|$ tends to infinity.
The calculation of the cardinality of $\lambda(X)$ seems to be a
difficult combinatorial problem, which can be reformulated as the problem of counting the number $\lambda(n)$ of  self-dual monotone Boolean functions of $n$ variables, which is well-known in Discrete Mathematics.
According to Proposition 1.1 in \cite{BMMV}, $$\log_2\lambda(n)=\frac{2^n}{\sqrt{2\pi n}}+o(1),$$ which means that the sequence $(\lambda(n))_{n=1}^\infty$ has double exponential growth.
The sequence of numbers  $\lambda(n)$ (known in Discrete Mathematics as Ho\c sten-Morris numbers) is included in the {\tt On-line Encyclopedia of Integer Sequences} as the sequence A001206. All known precise values of this sequence (taken from \cite{BMMV}) are presented the following table.

\begin{center}
\begin{tabular}{|c|ccccccccc|}
\hline
$|X|$ & 1 & 2 & 3 & 4 & 5 & 6 & 7 & 8 &9\\
\hline
$|\lambda(X)|$ & 1 & 2 & 4 & 12 & 81 & 2646 & 1422564 &229809982112&423295099074735261880\\
\hline
\end{tabular}
\end{center}
\bigskip

 Each map $f:X\to Y$ induces the map
$$\lambda f:\lambda (X)\to \lambda (Y),\quad \lambda f:\mathcal L\mapsto
\big\langle f(L)\subset Y: L\in\mathcal L\big\rangle, \ \text{ see \cite{G1}}.$$

If $\varphi: S\to S'$ is a homomorphism of semigroups, then $\lambda \varphi:
\lambda(S)\to \lambda(S')$ is a homomorphism as well, see \cite{G4}.

A non-empty subset $I$ of a semigroup $S$ is called an {\em  ideal}
if $IS\cup SI\subset I$. An ideal $I$ of a semigroup $S$ is called  {\em proper} if $I\neq S$.
A proper ideal $M$ of $S$ is  {\em maximal} if $M$ coincides with each proper ideal $I$ of $S$ that contains $M$.
It is easy to see that for every $n\in \IN$ the subset $S^{\cdot n}=\{x_1\cdot\ldots\cdot x_n:x_1,\dots,x_n\in S\}$ is an ideal in $S$.

An element $z$ of a semigroup $S$ is called a {\em zero} (resp. a {\em  left zero}, a {\em right
zero}) in $S$ if $az=za=z$ (resp. $za=z$, $az=z$) for any $a\in
S$. An element $e$ of a semigroup $S$ is called an {\em idempotent} if $ee=e$. By $E(S)$ we denote the
set of all idempotents of a semigroup $S$.

Recall that an {\em isomorphism} between $S$ and $S'$ is a bijective function
$\psi:S\to S'$ such that $\psi(xy)=\psi(x)\psi(y)$ for
all $x,y\in S$. If there exists an isomorphism between $S$ and
$S'$, then $S$ and $S'$ are said to be {\em isomorphic}, denoted
$S\cong S'$. An isomorphism $\psi:S\to S$ is called an {\em automorphism} of a semigroup $S$. By $\Aut(S)$ we denote
the automorphism group of a semigroup $S$.

For an automorphism $\psi$ of a semigroup $S$, a subset $A\subset S$ is called {\em $\psi$-invariant} if $\psi(A)=A$. A subset $A\subset S$ is called {\em characteristic} if $\psi(A)=A$ for any automorphism $\psi$ of $S$. It is easy to see that the set $E(S)$ is characteristic in $S$ and so are the ideals $S^{{\cdot} n}$ for all $n\in\IN$.

For a set $X$ by $S_X$ we denote the group of all bijections of $X$. For two sets $X\subset Y$ we shall identify $S_X$ with the subgroup $\{\varphi\in S_Y:\varphi|Y\setminus X=\mathrm{id}\}$ of the permutation group $S_Y$. By $C_n$ we denote a cyclic group of order $n\in\IN$.

In this paper we study automorphisms of superextensions of finite monogenic semigroups.
The thorough study of automorphism groups of  superextensions of semigroups was started in
\cite{G9} and continued in \cite{BG9}.
In \cite{G9} it was shown that each automorphism of a semigroup $S$ can be extended to an automorphism of its superextension $\lambda(S)$
and the automorphism group $\Aut(\lambda(S))$ of the superextension of a semigroup $S$ contains a subgroup,
isomorphic to the auto\-morphism group $\Aut(S)$ of $S$. Also  the automorphism groups of  superextensions
of null semigroups, almost null semigroups, right zero semigroups,  left zero semigroups and all  three-element semigroups were  described.
In \cite{BG9} we studied the automorphism groups  of super\-exten\-sions of groups and described  the automorphism groups $\Aut(\lambda(G))$
of the super\-exten\-sions of all groups $G$ of cardinality $|G|\leq 5$.
The obtained results  are presented in Table~\ref{tab:dedekind_gr}.

\begin{table}[ht]
\centering
\begin{tabular}{|r|cccccc|}
\hline
$G\cong$ & $C_1$ & $C_2$ & $C_3$ & $C_4$ & $C_2\times C_2$&$C_5$\\
\hline
$\Aut(G)\cong$ & $C_1$& $C_1$ & $C_2$ & $C_2$ & $S_3$ &$C_4$\\
\hline
$\Aut(\lambda(G))\cong$ & $C_1$ & $C_1$ & $C_2$ & $C_2\times C_2$ & $S_4$ &$C_4$\\
\hline
\end{tabular}
\smallskip
\caption{The automorphism groups of superextensions of groups of cardinality $\le 5$.}\label{tab:dedekind_gr}
\end{table}
\vskip-18pt

In Section~\ref{monsem} we establish some general results on the structure of superextensions of finite monogenic groups, using the notion of a good shift introduced in Section~\ref{s:good}. In particular, in Theorem~\ref{t3.2} we prove that two monogenic groups are isomorphic if and only if their superextensions are isomorphic. In Section~\ref{s4} we use the results of Section~\ref{monsem} to classify the authomorphism groups of the superextensions of monogenic semigroups of order $\le 5$. The obtained results are summed up in the table in Section~\ref{s:table}.

\section{Semigroups possessing a good shift}\label{s:good}

In this section we develop a tool for recognizing the automorphism group of a semigroup possessing a good shift.

\begin{definition}\label{d:good} Let $X$ be a semigroup. A function $\sigma:X\to X$ is called a {\em good shift} of $X$ if
$X\cdot X\subset \sigma(X)$ and for any elements $x,y\in X$ with $\sigma(x)=\sigma(y)$ the equalities $xz=yz$ and $zx=zy$ hold for all $z\in X$.

A good shift $\sigma:S\to S$ is called an {\em auto-shift} if $\sigma\circ\psi=\psi\circ\sigma$ for any automorphism $\psi\in\Aut(S)$.

\end{definition}

\begin{proposition}\label{p:kernel} If $\sigma:X\to X$ is a good shift of a semigroup $X$, then
the automorphism group $\Aut(X)$ of $X$ contains the subgroup $K$ consisting of all bijections $\psi:X\to X$ of $X$ such that  $\sigma\circ\psi=\sigma$ and $\psi|XX=\mathrm{id}$. This subgroup is isomorphic to $\prod_{x\in \sigma(X)}S_{\sigma^{-1}(x)\setminus XX}$.
\end{proposition}

\begin{proof} It is clear that the group $K$ is isomorphic to the product
of permutation groups $\prod_{x\in \sigma(X)}S_{\sigma^{-1}(x)\setminus XX}$.

To see that $K\subset\Aut(X)$, it is necessary to check that each bijection $\psi\in K$ is an  automorphism of the semigroup $X$.
Given two elements $x,y\in X$, we shall show that $\psi(xy)=\psi(x)\cdot\psi(y)$. It follows that $xy\in XX$ and hence $\psi(xy)=xy$. Since $\sigma\circ\psi=\sigma$, we obtain that $\sigma(\psi(x))=\sigma(x)$ and $\sigma(\psi(y))=\sigma(y)$. Now Definition~\ref{d:good} ensures that
$$\psi(x)\cdot\psi(y)=x\cdot\psi(y)=xy=\psi(xy),$$
so $\psi\in\Aut(X)$.
\end{proof}

A homomorphism $\rho:S\to S$ of a semigroup $S$ is called a {\em homomorphic retraction} if $\rho\circ\rho=\rho$. The functoriality of the superextension in the category of semigroups \cite{G4} ensures that for any homomorphic retraction $\rho:S\to S$ the map $\lambda\rho:\lambda(S)\to\lambda(S)$ is a homomorphic retraction, too.

For a semigroup $X$ by $\Xi_X$ denote the family of all characteristic subsets of $X$.

\begin{theorem}\label{t:Ker} If $\sigma:X\to X$ is an auto-shift of a semigroup $X$, then $\sigma(X)$ is a characteristic ideal of $X$, and for the restriction operator $$R:\Aut(X)\to\Aut(\sigma(X)),\;\; R:\psi\mapsto\psi|\sigma(X),$$
the kernel of $R$ is equal to $\prod_{x\in\sigma(X)}S_{\sigma^{-1}(x)\setminus\sigma(X)}$ and $R(\Aut(X))\subset H\subset G$ where  $$G=\big\{\varphi\in\Aut(\sigma(X)):\forall x\in\sigma(X)\;\;\varphi(\sigma^{-1}(x)\cap\sigma(X))=\sigma^{-1}(\varphi(x))\cap\sigma(X)\mbox{ \ and \ } |\sigma^{-1}(x)\setminus\sigma(X)|=|\sigma^{-1}(\varphi(x))\setminus\sigma(X)|\big\}$$
and
$$H:=\big\{\varphi\in G:\forall C\in\Xi_X\;\forall x\in\sigma(X)\;\;|\sigma^{-1}(x)\cap C|=|\sigma^{-1}(\varphi(x))\cap C|\big\}.$$If the auto-shift $\sigma$ is a homomorphic retraction, then $R(\Aut(X))=H=G$.
\end{theorem}

\begin{proof} Definition~\ref{d:good} ensures that $XX\subset \sigma(X)$, so $\sigma(X)$ is an ideal in $X$. Since $\sigma$ is an auto-shift, for any automorphism $\psi\in\Aut(X)$ we have $\psi\circ\sigma=\sigma\circ\psi$ and hence
$$\psi(\sigma(X))=\psi\circ\sigma(X)=\sigma\circ \psi(X)=
\sigma(X),$$
which means that the ideal $\sigma(X)$ is characteristic in $X$.

Now consider the group $K$ of all bijections $\psi:X\to X$ such that $\sigma\circ\psi=\sigma$ and $\psi|\sigma(X)$ is the identity map of $\sigma(X)$.
It is clear that the group $K$ is isomorphic to the product of permutation groups  $\prod_{x\in \sigma(x)}S_{\sigma^{-1}(x)\setminus\sigma(X)}$.
By Proposition~\ref{p:kernel}, $K\subset\Aut(X)$.

It is clear that $K$ is contained in the kernel $\mathrm{Ker}(R)$ of the restriction operator $R:\Aut(X)\to\Aut(\sigma(X))$. On the other hand, for any automorphism $\psi\in\mathrm{Ker}(R)$ we get $\sigma\circ\psi=\psi\circ\sigma=\sigma$, so $\psi\in K$.

Observe that for any automorphism $\psi$ of $X$ the equality $\psi^{-1}\circ\sigma=\sigma\circ\psi^{-1}$ implies that
$\sigma^{-1}(y)=\sigma^{-1}(\psi(y))$ for any $y\in\sigma(X)$.
Moreover, for any characteristic set $C\in\Xi_X$ we get $\psi(C)=C$ and hence
$$\psi(\sigma^{-1}(x)\cap C)=\psi(\sigma^{-1}(x))\cap \psi(C)=\sigma^{-1}(\psi(x))\cap C$$ and hence  $|\sigma^{-1}(x)\cap C|=|\sigma^{-1}(\psi(x))\cap C|$. Since $\sigma(X)\in\Xi_X$, for the restriction $\varphi:=\psi|\sigma(X)$ of $\psi$ we get $\varphi(\sigma^{-1}(x)\cap\sigma(X))=\sigma^{-1}(x)\cap\sigma(X)$ for every $x\in\sigma(X)$, which means that $\varphi=R(\psi)\in H$ and hence $R(\Aut(X))\subset H$.

The inclusion $H\subset G$ follows from $X\setminus \sigma(X)\in\Xi_X$.

Assuming that $\sigma$ is a homomorphic retraction, we shall prove that $R(\Aut(X))=H=G$.
It suffices to check that each automorphism $\varphi\in G$ extends to an automorphism of the semigroup $X$.
By the definition of $G$, we can extend the automorphism $\varphi$ of $\sigma(X)$ to a bijection $\bar\varphi$
of $X$ such that $\bar\varphi(\sigma^{-1}(x)\setminus\sigma(X))=\sigma^{-1}(\varphi(x))\setminus\sigma(X)$ for all $x\in\sigma(X)$.
By $\varphi\in G$, we also get $\varphi(\sigma^{-1}(x))\cap\sigma(X)=\sigma^{-1}(\varphi(x))\cap\sigma(X)$. So, $\bar\varphi(\sigma^{-1}(x))=\sigma^{-1}(\varphi(x))$ for any $x\in\sigma(X)$.
This equality implies the equality $\sigma\circ\bar\varphi=\varphi\circ\sigma$. Indeed, for any $z\in X$ and $x=\sigma(z)$ we get $\bar\varphi(z)\in\bar\varphi(\sigma^{-1}(x))=\sigma^{-1}(\varphi(x))$. Applying to this equality the map $\sigma$, we obtain $\sigma\circ\bar\varphi(z)\in\{\varphi(x)\}$ and hence $\sigma\circ\bar\varphi(z)=\varphi(x)=\varphi\circ\sigma(z)$.

Let us show that the bijection $\bar\varphi$ is an automorphism of $X$. Given any elements $x,y\in X$, it suffices to check that $\bar\varphi(xy)=\bar\varphi(x)\cdot\bar\varphi(y)$.
Taking into account that $xy\in XX\subset\sigma(X)$ and $\sigma:X\to\sigma(X)$ is a homomorphic retraction, we conclude that $xy=\sigma(xy)=\sigma(x)\cdot\sigma(y)$ and thus $$
\bar\varphi(xy)=\varphi(xy)=\varphi(\sigma(xy))=\varphi(\sigma(x)\cdot\sigma(y))=
\varphi(\sigma(x))\cdot\varphi(\sigma(y))=\sigma(\bar\varphi(x))\cdot
\sigma(\bar\varphi(y))=\sigma(\bar\varphi(x)\cdot\bar\varphi(y))=
\bar\varphi(x)\cdot
\bar\varphi(y).
$$
\end{proof}

\section{Automorphisms and characteristic ideals of superextensions of finite monogenic semigroups}\label{monsem}

A semigroup $S$ is called {\em monogenic} if it is generated by some element $a\in S$ in the sense that $S=\{a^n\}_{n\in\mathbb N}$. If a monogenic semigroup is
infinite, then it is isomorphic to the additive semigroup $\mathbb N$ of positive integer numbers. A
finite monogenic semigroup $S=\langle a\rangle$ also has simple
structure, see  \cite{Howie}. There are positive integer numbers $r$
and $m$ called the {\em index} and the {\em period} of $S$ such
that
\begin{itemize}
\item $S=\{a,a^2,\dots,a^{r+m-1}\}$ and $r+m-1=|S|$;
\item $a^{r+m}=a^{r}$;
\item $C_m:=\{a^r,a^{r+1},\dots,a^{r+m-1}\}$ is
 a cyclic and maximal subgroup of $S$ with the
neutral element $e=a^n\in C_m$ and generator $a^{n+1}$, where $n\in (m\cdot\mathbb N)\cap\{r,\dots,r+m-1\}$.
\end{itemize}

From now on, we denote by $\M_{r,m}$ a finite monogenic semigroup of
index $r$ and period $m$. Consider the shift
$$\sigma:\M_{r,m}\to a\M_{r,m},\;\;\sigma:x\mapsto ax,$$and observe that for
every $k\in\IN$ the shift $\sigma^k:\M_{r,m}\to a^k\M_{r,m}$, $\sigma^k:x\mapsto a^kx$, coincides with the $k$th iteration of $\sigma$.

Observe also that the subset $a^k\M_{r,m}=\sigma^k(\M_{r,m})$ coincides with
the characteristic ideal $\M_{r,m}^{\cdot (k+1)}$ of $\M_{r,m}$. For $k\ge r$
the characteristic ideal $\M_{r,m}^{\cdot k}=a^{k-1}\M_{r,m}$ coincides with the minimal
ideal of $\M_{r,m}$ and with a maximal subgroup $C_m$ of $\M_{r,m}$.

For the idempotent $e=a^n$ of the group
$C_m$ the shift $\sigma^n:\M_{r,m}\to a^n\M_{r,m}=C_m$ is a homomorphic retraction of $\M_{r,m}$ onto
its maximal subgroup $C_m$. This retraction will be denoted by $\rho$.

The following lemma distinguishes some characteristic ideals in the semigroup $\lambda(\M_{r,m})$.

\begin{lemma}\label{l:char-i} For every $k\ge 2$ the subsemigroup $\lambda(\M^{\cdot k}_{r,m})$ of $\lambda(\M_{r,m})$
coincides with the characteristic ideal  $\lambda(\M_{r,m})^{\cdot k}$ of the semigroup $\lambda(\M_{r,m})$.
\end{lemma}

\begin{proof} By the functoriality of the superextension, the surjective map
$$\sigma^{k-1}:\M_{r,m}\to a^{k-1}\M_{r,m}=\M_{r,m}^{\cdot k},\;\;\sigma^{k-1}:x\mapsto a^{k-1}x,$$
induces the surjective map $\lambda \sigma^{k-1}:\lambda(\M_{r,m})\to\lambda(\M_{r,m}^{\cdot k})$, $\lambda\sigma^{k-1}:\A\mapsto a^{k-1}*\A$.
Then for every maximal linked upfamily $\mathcal B\in\lambda(a^{k-1}\M_{r,m})$ we can find a maximal linked upfamily $\A\in\lambda(\M_{k,m})$ with $\mathcal B=a^{k-1}*\A$ and conclude that $
\mathcal B=a^{k-1}*\A\in\lambda(\M_{k,m})^{\cdot k}$.

On the other hand, the definition of the semigroup operation on $\lambda(\M_{r,m})$
implies that $\lambda(\M_{r,m})^{\cdot k}\subset \lambda(\M_{r,m}^{\cdot k})= \lambda(a^{k-1}\M_{r,m})$.
So, $\lambda(\M_{r,m}^{\cdot k})=\lambda(\M_{r,m})^{\cdot k}$ is a characteristic ideal in $\lambda(\M_{r,m})$.
\end{proof}

\begin{theorem}\label{t3.2} Two finite monogenic semigroups are isomorphic if and only if their superextensions are isomorphic.
\end{theorem}

\begin{proof} If two semigroups are isomorphic, then their superextensions are isomorphic by the functoriality of the superextension in the category of semigroups \cite{G4}. Now assume that two monogenic semigroups $\M_{i,m}$ and $\M_{j,n}$ have isomorphic superextensions. Let $\psi:\lambda(\M_{i,m})\to\lambda(\M_{j,n})$ be an isomorphism.

Let $C_m$ be the maximal subgroup of $\M_{i,m}$. By Lemma~\ref{l:char-i}, the superextension $\lambda(C_m)$ is a characteristic subsemigroup of $\lambda(\M_{i,m})$, equal to the intersection $\bigcap_{k\in\IN}\lambda(\M_{i,m})^{\cdot k}$ of the characteristic ideals $\lambda(\M_{i,m})^{\cdot k}$. Then
$$\psi(\lambda(C_m))=\psi\big({\textstyle \bigcap}_{k\in\IN}\lambda(\M_{i,m})^{\cdot k}\big)=\bigcap_{k\in\IN}\psi\big(\lambda(\M_{i,m})^{\cdot k}\big)=
\bigcap_{k\in\IN}\big(\psi(\lambda(\M_{i,m}))\big)^{\cdot k}=
\bigcap_{k\in\IN}\lambda(\M_{j,n})^{\cdot k}=
\lambda(C_n)$$and hence
$|\lambda(C_m)|=|\lambda(C_n)|$.

Observe that two finite sets $X,Y$ have the same cardinality if $|\lambda(X)|=|\lambda(Y)|$.
Indeed, assuming that $|X|<|Y|$ we can choose an injective map $f:X\to Y$ with $f(X)\ne Y$ and obtain an injective
map $\lambda f:\lambda(X)\to\lambda(Y)$ with $\lambda f(\lambda(X))\ne\lambda(Y)$, which implies that $|\lambda(X)|<|\lambda(Y)|$. Now we see that the equality $|\lambda(C_m)|=|\lambda(C_n)|$ implies $n=m$.

On the other hand, the equality $|\lambda(\M_{i,m})|=|\lambda(\M_{j,n})|$ implies $i+m-1=|\M_{i,m}|=|\M_{j,n}|=j+n-1$. Since $m=n$, we obtain that $i=j$ and hence $\M_{i,n}=\M_{j,m}$.
\end{proof}

\begin{proposition}\label{mono1} If $r\ge 3$, then any automorphism $\psi$ of the semigroup $\lambda(\M_{r,m})$ has $\psi(x)=x$ for all $x\in \M_{r,m}$.
\end{proposition}

\begin{proof} Let $a$ be the (unique) generator of the monogenic semigroup $\M_{r,m}$.
By Lemma~\ref{l:char-i}, $a*\lambda(\M_{r,m})=\lambda(\M_{r,m})^{\cdot 2}$ is a characteristic
ideal in $\lambda(\M_{r,m})$. Then for any automorphism $\psi\in\Aut(\lambda(\M_{r,m}))$ we get
 $$\psi(a)*\lambda(\M_{r,m})=\psi(a)*\psi(\lambda(\M_{r,m}))=\psi(a*\lambda(\M_{r,m}))=a*\lambda(\M_{r,m})\ni a^2$$ and
 hence $a^2=\psi(a)*\A$ for some $\A\in\lambda(\M_{r,m})$. Taking into account that for $r\ge 3$, the equality $a^2=x*y$ has a unique solution $x=y=a$ in $\lambda(\M_{r,m})$, we conclude that $\psi(a)=a$ and hence $\psi(x)=x$ for all $x\in\M_{r,m}$.
\end{proof}

In the following proposition by $e$ we denote the unique idempotent of the
group $C_m\subset\M_{r,m}$ and by $\rho:\M_{r,m}\to C_m$, $\rho:x\mapsto ex$, the homomorphic
retraction of $\M_{r,m}$ onto $C_m$. The homomorphic retraction $\rho$ induces a homomorphic
retraction $\bar\rho:\lambda(\M_{r,m})\to\lambda(C_m)$, $\bar\rho:\A\mapsto e*\A$.

\begin{theorem}\label{t:r=2} For $r=2$ the homomorphic retraction $\bar\rho:\lambda(\M_{r,m})\to\lambda(C_m)\subset\lambda(\M_{r,m})$ has the following properties:
\begin{enumerate}
\item $\A*\mathcal B=\bar\rho(\A)*\mathcal B=\A*\bar\rho(\mathcal B)=\bar\rho(\A)*\bar\rho(\mathcal B)$ for any $\A,\mathcal B\in\lambda(\M_{r,m})$;
\item $\psi(x)=x$ for any $x\in C_m$ and any $\psi\in\Aut(\lambda(\M_{r,m}))$;
\item the homomorphic retraction $\bar\rho$ is an auto-good shift of $\lambda(\M_{r,m})$;
\item the operator $R:\Aut(\lambda(\M_{r,m}))\to\Aut(\lambda(C_m))$ has kernel isomorphic to $\prod_{\mathcal L\in\lambda(C_m)}S_{\bar\rho^{-1}(\mathcal L)\setminus\{\mathcal L\}}$ and the range $R(\Aut(\M_{r,m}))=\{\varphi\in\Aut(\lambda(C_m)):\forall \mathcal L\in\lambda(C_m)\;\;|\bar\rho^{-1}(\varphi(\mathcal L))|=|\bar\rho^{-1}(\mathcal L)|\}$.
\end{enumerate}
\end{theorem}

\begin{proof}
 1. The first statement follows from the definition of the semigroup operation on $\lambda(\M_{r,m})$ and the equality $xy=\rho(x)\cdot y=x\cdot\rho(y)=\rho(x\cdot y)$ holding for any elements $x,y\in\M_{r,m}=\M_{2,m}$.
\smallskip

2. Fix any automorphism $\psi$ of the semigroup $\lambda(\M_{r,m})$.
By Lemma~\ref{l:char-i}, $\lambda(C_m)$ is a characteristic ideal of $\lambda(\M_{r,m})$,
so the restriction $\psi|\lambda(C_m)$ is an automorphism of the semigroup $\lambda(C_m)$. In \cite{BG9} it was proved that $\psi(C_m)=C_m$.
Taking into account that $e$ is a unique idempotent of $C_m$, we conclude that $\psi(e)=e$.

Since $\psi$ is an automorphism of $\lambda(\M_{r,m})$, for every $\A\in\lambda(\M_{r,m})$ we have
$$\psi(\bar\rho(\A))=\psi(e*\A)=\psi(e)*\psi(\A)=e*\psi(\A)=\bar\rho(\psi(\A)),$$
which implies that $\psi(\bar\rho^{-1}(\mathcal B))=\bar\rho^{-1}(\psi(\mathcal B))$ and hence
$|\bar\rho^{-1}(\mathcal B)|=|\bar\rho^{-1}(\psi(\mathcal B))|$ for all $\mathcal B\in\lambda(C_m)$. In particular,
for the generator $a$ of the semigroup $\M_{r,m}$ we get $|\bar\rho^{-1}(ae)|=|\bar\rho^{-1}(\psi(ae))|$. Now observe that
$\bar\rho^{-1}(ae)\supset\rho^{-1}(ae)=\{ae,a\}$ and $\bar\rho^{-1}(x)=\rho^{-1}(x)=\{x\}$ for any $x\in C_m\setminus\{ae\}$. This implies that $\psi(ae)=ae$ and hence $\psi(x)=x$ for all $x\in C_m$ (as $ae$ is a generator of the group $C_m$).
\smallskip

3. The third statement follows from the preceding two statements.

4. The fourth statement follows from the preceding statement and Theorem~\ref{t:Ker}.
\end{proof}

\begin{lemma}\label{l:auto-shift} If $r\ge 2$, then the map $$\bar\sigma:\lambda(\M_{r,m})\to\lambda(\M_{r,m}^{\cdot2}),\;\;\bar\sigma:\A\mapsto a*\A$$ is an auto-shift of the semigroup $\lambda(\M_{r,m})$.
\end{lemma}

\begin{proof}
It is clear that $\bar\sigma(\lambda(\M_{r,m}))=a*\lambda(\M_{r,m})=\lambda(\M_{r,m})^{\cdot 2}$. Next, we show that $\bar\sigma\circ\psi=\psi\circ \bar\sigma$ for every automorphism $\psi$ of $\lambda(\M_{r,m})$.

If $r\ge 3$, then $\psi(a)=a$ by Proposition~\ref{mono1} and hence $$\psi\circ\bar\sigma(\A)=\psi(a*\A)=\psi(a)*\psi(\A)=a*\psi(\A)=\bar\sigma\circ\psi(\A)$$for any $\A\in\lambda(\M_{r,m})$.

If $r=2$, then $ax=aex$ for any $x\in\M_{r,m}$. By Theorem~\ref{t:r=2}, $\psi(ae)=ae$ and then
$$\psi\circ\bar\sigma(\A)=\psi(a*\A)=\psi(ae*\A)=\psi(ae)*\psi(\A)=ae*\psi(\A)=a*\psi(\A)=
\bar\sigma\circ\psi(\A)$$for any $\A\in\lambda(\M_{r,m})$.

It remains to prove that for any $\A,\A'\in\lambda(\M_{r,m})$ with $\bar\sigma(\A)=\bar\sigma(\A')$ the equalities $\A*\mathcal B=\A'*\mathcal B$ and $\mathcal B*\A=\mathcal B*\A'$ hold for all $\mathcal B\in\lambda(\M_{r,m})$.

It follows from $a*\A=\bar\sigma(\A)=\bar\sigma(\A')=a*\A'$ that $a^n*\A=a^n*\A'$ for all $n\in\IN$ and hence $x*\A=x*\A'$
for all $x\in\M_{r,m}$. To see that $\mathcal B*\A=\mathcal B*\A'$, take any set $C\in\mathcal B*\A$ and find
a set $B\in\mathcal B$ and a family $(A_b)_{b\in B}\in \A^B$ such that $\bigcup_{b\in B}bA_b\subset C$.
For every $b\in B$, we can use the equality $bA_b\in b*\A=b*\A'$ to find a set $A'_b\in \A'$ such that $bA_b'\subset bA_b$.
Then $$\mathcal B*\A'\ni\bigcup_{b\in B}bA_b'\subset\bigcup_{b\in B}bA_b\subset C$$and hence $C\in\mathcal B*\A'$.
By analogy we can prove that $\mathcal B*\A'\subset\mathcal B*\A$.

Next, we prove that $\A*\mathcal B=\A'*\mathcal B$. Given any element $C\in\A*\mathcal B$, find a set $A\in\A$ and a family $\{B_x\}_{x\in A}\subset\mathcal B$ such that $\bigcup_{x\in A}xB_x\subset C$. Since $aA\in a*\A=a*\A'$, there exists a set $A'\in\A'$ such that $aA'\subset aA$. For any $y\in A'$ find $x\in A$ with $ay=ax$ and put $B_{y}:=B_x$. We claim that $\bigcup_{y\in A'}yB_y\subset C$. Indeed, for any $z\in \bigcup_{y\in A'}yB_y$ we can find $y\in A'$ and $b\in B_y$ such that $z=yb\in yB_y$. For the element $y$ there exists an element $x\in A$ such that $ax=ay$ and $B_y=B_x$. The equality $ax=ay$ implies $a^kx=a^ky$ for all $k\in\IN$. In particular, $bx=by$. Then $z=yb=by=bx=xb\in xB_x\subset C$ and hence $\A'*\mathcal B\ni\bigcup_{y\in A'}B_y\subset C$, which implies $C\in\A'*\mathcal B$ and $\A*\mathcal B\subset \A'*\mathcal B$. By analogy we can prove that $\A'*\mathcal B\subset \A*\mathcal B$ and thus $\A'*\mathcal B=\A*\mathcal B$.
\end{proof}

\begin{corollary}\label{c:good} If $r\ge 2$, then for any $k\in\IN$ the map $$\bar\sigma_k:\lambda(\M^{\cdot k}_{r,m})\to\lambda(\M_{r,m}^{\cdot (k+1)}),\;\;\bar\sigma_k:\A\mapsto a*\A,$$ is a good shift of the semigroup $\lambda(\M_{r,m}^{\cdot k})$.
\end{corollary}

Lemma~\ref{l:auto-shift} and Theorem~\ref{t:Ker} have the following:

\begin{corollary}\label{c:Ker} Assume that $r\ge 2$. The operator $R:\Aut(\lambda(\M_{r,m}))\to\Aut(\lambda(\M_{r,m}^{\cdot2}))$
has kernel isomorphic to $\prod_{\mathcal L\in\lambda(\M_{r,m}^{\cdot 2})}S_{\bar\sigma^{-1}(\mathcal L)\setminus\lambda(\M_{r,m}^{\cdot2})}$
and range $R(\Aut(\M_{r,m}))\subset H$ where
$$
\begin{aligned}
H=\big\{\varphi\in\Aut(\lambda(\M_{r,m}^{\cdot 2})):\ &\forall \mathcal L\in\lambda(\M_{r,m}^{\cdot 2})\;\;\varphi(\bar\sigma^{-1}(\mathcal L)\cap\lambda(\M_{r,m}^{\cdot2}))=\bar\sigma^{-1}(\mathcal L)\cap\lambda(\M_{r,m}^{\cdot2})\mbox{ and }\\
&\forall C\in\Xi_{\lambda(\M_{r,m})}\;|\bar\sigma^{-1}(\varphi(\mathcal L))\cap C|=|\bar\sigma^{-1}(\mathcal L)\cap C|\big\}.
\end{aligned}
$$
\end{corollary}

\section{Automorphism groups of superextensions of monogenic semigroups of cardinality $\le 5$}\label{s4}

In this section we shall describe the structure of the automorphism groups of superextensions of all monogenic semigroups $\M_{r,m}$ of cardinality $|\M_{r,m}|\le 5$.

\subsection{The semigroups $\lambda(\M_{1,1})$, $\lambda(\M_{1,2})$ and $\lambda(\M_{2,1})$} For any $r,m\in\IN$ with $r+m\le 3$, the monogenic semigroup $\M_{r,m}$ has cardinality $|\M_{r,m}|\le 2$. Consequently, $\lambda(\M_{r,m})=\M_{r,m}$ and $\Aut(\lambda(\M_{r,m}))=\Aut(\M_{r,m})\cong C_1$.

\subsection{The semigroups $\lambda(\M_{1,3})$, $\lambda(\M_{2,2})$ and $\lambda(\M_{3,1})$}
For  a monogenic semigroup $\M_{r,m}$ with $m+r=4$ and generator $a$, the superextension $\lambda(\M_{r,m})$  consists of three principal ultrafilters $a,a^2,a^3$
 and the maximal linked family
$\triangle=\langle\{a,a^2\},\{a,a^3\},\{a^2,a^3\}\rangle$.

Taking into account that $\M_{1,3}$ is isomorphic to $C_3$ and $\Aut(\lambda(C_3))\cong C_2$ (see Table~\ref{tab:dedekind_gr}), we conclude that $\Aut(\lambda(\M_{1,3}))\cong\Aut(\lambda(C_3))\cong C_2$.

By Proposition~\ref{mono1}, for any $i\in\{1,2,3\}$  and $\psi\in\Aut(\lambda(\M_{3,1}))$
we have $\psi(a^i)=a^i$. Then  $\psi(\triangle)=\triangle$ and hence $\Aut(\lambda(\M_{3,1}))\cong C_1$.

To recognize the automorphism group of $\lambda(\M_{2,2})$, consider the homomorphic
retraction $$\bar\rho:\lambda(\M_{2,2})\to\lambda(\M_{2,2}^{\cdot2})=\{a^2,a^3\},\;\;\bar\rho:\A\mapsto a^2\A$$and
observe that it has fibers: $\bar\rho^{-1}(a^2)=\{a^2\}$ and $\bar\rho^{-1}(a^3)=\{a,a^3,\triangle\}$.
Since the automorphism group $\Aut(\lambda(\M_{2,2}^{\cdot 2}))\cong\Aut(\lambda (C_2))$ is trivial, we can apply Theorem~\ref{t:r=2}
and conclude that the automorphism group $\Aut(\lambda(\M_{2,2}))$ is isomorphic to $S_{\bar\rho^{-1}(a^3)\setminus\{a^3\}}=S_{\{a,\triangle\}}\cong C_2$.

\subsection{The semigroup $\lambda(\M_{1,4})$}
The semigroup $M_{1,4}$ is isomorphic to the cyclic group $C_4$. Therefore, $$\Aut(\lambda(\M_{1,4}))\cong\Aut(\lambda(C_4))\cong C_2\times C_2,$$
according to \cite{BG9} (see also Table~\ref{tab:dedekind_gr}).

\subsection{The semigroup $\lambda(\M_{2,3})$}\label{ss:M23}
Consider the semigroup $\M_{2,3}=\{a,a^2,a^3,a^4\ |\ a^5=a^2\}$ generated by the element $a$. Its superextension $\lambda(\M_{2,3})$ contains $12$ elements of the form:
$$a^i,\;\;\triangle_i=\langle A:A\subset \M_{2,3}\setminus\{a^i\},\;|A|=2\rangle,\mbox{ and }\;\;
\square_i=\langle\M_{2,3}\setminus\{a^i\},\;\{a^i,x\}:x\in\M_{2,3}\setminus\{a^i\}\rangle$$
where $i\in\{1,2,3,4\}$.

Let $e=a^3$ be the idempotent of the subgroup $C_3$ and $\rho:\M_{2,3}\to C_3$, $\rho:x\mapsto ex$,
be the homomorphic retraction of $\M_{2,3}$ onto $C_3$. The map $\rho$ induces a homomorphic
retraction $\bar\rho:\lambda(\M_{2,3})\to\lambda(C_3)=\{a^2,a^3,a^4,\triangle_1\}$, defined by $\bar\rho(\A)=a^3*\A$ for $\A\in\lambda(\M_{2,3})$.
By routine calculations we can show that the map $\bar\rho$ has fibers:
\begin{itemize}
\item $\bar\rho^{-1}(a^2)=\{a^2\}$;
\item $\bar\rho^{-1}(a^3)=\{a^3\}$;
\item $\bar\rho^{-1}(a^4)=\{a,\;a^4,\triangle_2,\triangle_3,\square_1,\square_4\}$;
\item $\bar\rho^{-1}(\triangle_1)=\{\triangle_1,\;\triangle_4,\square_2,\square_3\}$.
\end{itemize}

Applying Theorem~\ref{t:r=2}, we can show that the automorphism group $\Aut(\lambda(\M_{2,3}))$ of $\lambda(\M_{2,3})$ is isomorphic to
$$S_{(\bar\rho)^{-1}(\triangle_1)\setminus\{\triangle_1\}}\times S_{(\bar\rho)^{-1}(a^4)\setminus\{a^4\}}\cong S_3\times S_5.$$

\subsection{The semigroup $\lambda(\M_{3,2})$}

In this section we consider the monogenic semigroup $\M_{3,2}=\{a,a^2,a^3,a^4\}$ generated by an element $a$ such that $a^5=a^3$. The semigroup $\M_{3,2}$ contains the characteristic ideals $\M_{3,2}^{\cdot 2}=\{a^2,a^3,a^4\}$ and $\M_{3,2}^{\cdot3}=\{a^3,a^4\}=C_2$.

Using the notations from the preceding subsection, we can see that the superextensions of the semigroups $\M_{3,2}^{\cdot k}$ contain the following elements:
\begin{itemize}
\item $\lambda(\M_{3,2})=\{a^i,\triangle_i,\square_i:1\le i\le 4\}$;
\item $\lambda(\M_{3,2}^{\cdot 2})=\{a^2,a^3,a^4,\triangle_1\}$;
\item $\lambda(\M_{3,2}^{\cdot3})=\lambda(C_2)=C_2=\{a^3,a^4\}$.
\end{itemize}

By Lemma~\ref{l:auto-shift}, the map $\bar\sigma:\lambda(\M_{3,2})\to\lambda(\M_{3,2}^{\cdot2})$, \ $\bar\sigma:\A\mapsto a*\A$, is an auto-shift of the semigroup $\lambda(\M_{3,2})$.

By  routine calculations it can be shown that the map $\bar\sigma$ has the following  fibers:
\begin{itemize}
\item $\bar\sigma^{-1}(a^2)=\{a\}$;
\item $\bar\sigma^{-1}(a^3)=\{a^2,a^4,\triangle_1,\triangle_3,\square_2,\square_4\}$;
\item $\bar\sigma^{-1}(a^4)=\{a^3\}$;
\item $\bar\sigma^{-1}(\triangle_1)=\{\triangle_2,\triangle_4,\square_1,\square_3\}$.
\end{itemize}

By Corollary~\ref{c:good}, the restriction $\bar\sigma_2=\bar\sigma|\lambda(\M_{3,2}^{\cdot 2})$ is a good shift of the
semigroup $\lambda(\M_{3,2}^{\cdot 2})$. This shift has a unique fiber $\bar\sigma_2^{-1}(a^3)=\{a^4,a^2,\triangle_1\}$
containing more than one point. By Proposition~\ref{p:kernel}, the automorphism group $\Aut(\lambda(\M_{3,2}^{\cdot2}))$ contains
the permutation group $S_{\{a^2,\triangle_1\}}\cong C_2$. Taking into account that $C_2$ is a characteristic subgroup
of $\lambda(\M_{3,2}^{\cdot2})$ with trivial automorphism group, we conclude that $\Aut(\lambda(\M_{3,2}^{\cdot2}))=S_{\{a^2,\triangle_1\}}\cong C_2$.

By Corollary~\ref{c:Ker}, the restriction operator $R:\Aut(\lambda(\M_{3,2}))\to\Aut(\lambda(\M_{3,2}^{\cdot2}))\cong C_2$ has kernel isomorphic to
$$S_{\bar\sigma^{-1}(a^3)\setminus\lambda(\M_{3,2}^{\cdot2})}\times S_{\bar\sigma^{-1}(\triangle_1)\setminus\lambda(\M_{3,2}^{\cdot2})}\cong S_3\times S_4.$$
We claim that the subgroup $R(\Aut(\lambda(\M_{3,2})))$ of $\Aut(\lambda(\M_{3,2}^{\cdot2}))=S_{\{a^2,\triangle_1\}}$ is trivial.
Indeed, each automorphism $\varphi\in R(\Aut(\lambda(\M_{3,2})))$ is trivial on the subgroup
$C_2=\{a^3,a^4\}$, so $\varphi(a^2)\in\{a^2,\triangle_1\}$. By Corollary~\ref{c:Ker},
$$1=|\bar\sigma^{-1}(a^2)|=|\bar\sigma^{-1}(\varphi(a^2))|<4=|\bar\sigma^{-1}
(\triangle_1)|,$$ which implies that $\varphi(a^2)=a^2$ and $\varphi$ is the identity automorphism of $\Aut(\lambda(\M_{3,2}^{\cdot2}))$.
Consequently, the range of the restriction operator $R$ is trivial and the group $\Aut(\lambda(\M_{3,2}))$
coincides with the kernel of $R$ and hence is isomorphic to $S_3\times S_4$.

\subsection{The semigroup $\lambda(\M_{4,1})$}

In this subsection we consider the monogenic semigroup $\M_{4,1}=\{a,a^2,a^3,a^4\}$ generated by an element $a$ such that $a^5=a^4$.
The semigroup $\M_{4,1}$ contains the characteristic ideals $\M_{4,1}^{\cdot 2}=\{a^2,a^3,a^4\}$,  $\M_{4,1}^{\cdot3}=\{a^3,a^4\}$, and $\M_{4,1}^{\cdot 4}=\{a^4\}$.

Using the notations from the preceding subsection, we can see that the superextensions of the semigroups $\M_{4,1}^{\cdot k}$ contain the following elements:
\begin{itemize}
\item $\lambda(\M_{4,1})=\{a^i,\triangle_i,\square_i:1\le i\le 4\}$;
\item $\lambda(\M_{4,1}^{\cdot 2})=\{a^2,a^3,a^4,\triangle_1\}$;
\item $\lambda(\M_{4,1}^{\cdot3})=\M_{4,1}^{\cdot 3}=\{a^3,a^4\}$;
\item $\lambda(\M_{4,1}^{\cdot4})=\M_{4,1}^{\cdot4}=\{a^4\}$.
\end{itemize}
It is clear that the semigroups $\lambda(\M_{4,1}^{\cdot4})$ and $\lambda(\M_{4,1}^{\cdot3})$ have trivial automorphism groups.

Taking into account that $xy=a^4$ for any $x,y\in\lambda(\M_{4,1}^{\cdot2})=\{a^2,a^3,a^4,\triangle_1\}$, we conclude that the automorphism group $\Aut(\M_{4,1}^{\cdot2})$ of the semigroup $\M_{4,1}^{\cdot 2}$ is isomorphic to the symmetric group $S_3$.

By Lemma~\ref{l:auto-shift}, the map $$\bar\sigma:\lambda(\M_{4,1})\to\lambda(\M_{4,1}^{\cdot2}),\;\;\bar\sigma:\A\mapsto a*\A$$
is an auto-shift of the semigroup $\lambda(\M_{4,1})$.
By  routine calculations it can be shown that the map $\bar\sigma$ has the following fibers:
\begin{itemize}
\item $\bar\sigma^{-1}(a^2)=\{a\}$;
\item $\bar\sigma^{-1}(a^3)=\{a^2\}$;
\item $\bar\sigma^{-1}(a^4)=\{a^3,a^4,\triangle_1,\triangle_2,\square_3,\square_4\}$;
\item $\bar\sigma^{-1}(\triangle_1)=\{\triangle_3,\triangle_4,\square_1,\square_2\}$.
\end{itemize}

By Corollary~\ref{c:Ker}, the restriction operator $R:\Aut(\lambda(\M_{4,1}))\to\Aut(\lambda(\M_{4,1}^{\cdot2}))$ has kernel isomorphic to
$$S_{\bar\sigma^{-1}(a^4)\setminus\lambda(\M_{4,1}^{\cdot2})}\times S_{\bar\sigma^{-1}(\triangle_1)\setminus\lambda(\M_{4,1}^{\cdot2})}=S_{\{\triangle_2,\square_3,\square_4\}}\times S_{\{\triangle_3,\triangle_4,\square_1,\square_2\}}\cong S_3\times S_4.$$
By Proposition~\ref{mono1}, for any automorphism $\psi$ of $\lambda(\M_{4,1})$,
we get $\psi(a^k)=a^k$ for any $k\in\IN$. Taking into account that $\lambda(\M_{4,1}^{\cdot2})=\{a^2,a^3,a^4,\triangle_1\}$ is
a characteristic ideal in $\lambda(\M_{4,1})$, we conclude that $\psi(\triangle_1)=\triangle_1$.
Consequently, the range of the restriction operator $R$ is trivial and the group $\Aut(\lambda(\M_{4,1}))$
coincides with the kernel of $R$ and hence is isomorphic to $S_3\times S_4$.

\subsection{The semigroup $\lambda(\M_{1,5})$} The monogenic semigroup $\M_{1,5}$ is isomorphic to the cyclic group $C_5$.
It was proved in~\cite{BG9} that the automorphism group $\Aut(\lambda(\M_{1,5}))\cong\Aut(\lambda(C_5))$ is isomorphic to $C_4$.

\subsection{The semigroup $\lambda(\M_{2,4})$}\label{ss:M24} To describe the structure of the automorphism group $\Aut(\lambda(\M_{2,4}))$, we shall apply Theorem~\ref{t:r=2}.

Consider the homomorphic retraction
$$\bar\rho:\lambda(\M_{2,4})\to\lambda(\M_{2,4}^{\cdot 2}),\;\bar\rho:\mathcal L\mapsto e*\mathcal L,$$ where $e=a^4\in C_4\subset\M_{2,4}$ is the unique idempotent of the semigroup $\M_{2,4}$.

To describe the fibers of the retraction $\bar\rho$ we first introduce a convenient notation for all
81 elements of the superextension $\lambda(\M_{2,4})$. We recall that $\M_{2,4}=\{a,a^2,a^3,a^4,a^5\}$ and $C_4=\{a^2,a^3,a^4,a^5\}$ where $a^6=a^2$.

The 81 elements of the semigroup $\lambda(\M_{2,4})$ will be denoted as follows:
\begin{itemize}
\item $a^i$ for $i\in\{1,2,3,4,5\}$;
\item $\bigcirc:=\{A\subset \M_{2,4}:|A|\ge3\}$;
\item $\Theta_{ij}:=\langle\{a^i,a^j\},A:A\subset \M_{2,4},\;|A|=3,\;|A\cap\{a^i,a^j\}|=1\}\rangle$ for $1\le i<j\le 5$;
\item $\triangle_{ijk}:=\langle \{a^i,a^j\},\{a^i,a^k\},\{a^j,a^k\}\rangle$ for numbers $1\le i<j<k\le 5$;
\item $\Lambda^i:=\langle \M_{2,4}\setminus\{a^i\}, \{a^i,x\}:x\in \M_{2,4}\setminus\{a^i\}\rangle$ for $i\in\{1,2,3,4,5\}$;
\item $\diamondsuit^{n}_{ijk}:=\langle\{a^n,a^i\},\{a^n,a^j\},\{a^n,a^k\},\{a^i,a^j,a^k\}
    \rangle$ for numbers $1\le i<j<k\le 5$ and $n\in\{1,2,3,4,5\}\setminus \{i,j,k\}$;
\item $\Lambda^i_{jk}:=\big\langle\{a^i,a^j\},\{a^i,a^k\},A:A\subset \M_{2,4},\;|A|=3,\;A\cap\{a^i,a^j,a^k\}\in\big\{\{a^i\},\{a^j,a^k\}\big\}\big\rangle$ for $1\le j<k\le 5$ and $i\in \{1,2,3,4,5\}\setminus \{j,k\}$.
\end{itemize}
Observe that for the subgroup $C_4=\{a^2,a^3,a^4,a^5\}\subset\M_{2,4}$ we have
$$\lambda(C_4)=C_4\cup\{\triangle_{234},\triangle_{235},\triangle_{245},\triangle_{345},\diamondsuit^{2}_{345},\diamondsuit^3_{245},\diamondsuit^4_{235},\diamondsuit^5_{234}\}.$$
By routine calculations it can be shown that the map $\bar\rho:\lambda(\M_{2,4})\to\lambda(C_4)$ has the following fibers:
\begin{itemize}
\item $(\bar\rho)^{-1}(x)=\{x\}$ for $x\in\{a^2,a^3,a^4,\triangle_{234}\}$;
\item $(\bar\rho)^{-1}(a^5)=\{a^5,\;a,\triangle_{125},\triangle_{135},\triangle_{145},\diamondsuit^1_{235},
\diamondsuit^1_{245},\diamondsuit^1_{345},\diamondsuit^5_{123},\diamondsuit^5_{124}, \diamondsuit^5_{134},\Lambda^1,\Lambda^5,\Theta_{15},\Lambda^1_{25}, \Lambda^1_{35}, \Lambda^1_{45}, \Lambda^5_{12},\Lambda^5_{13},\Lambda^5_{14}\}$;
\item $(\bar\rho)^{-1}(\triangle_{235})=\{\triangle_{235},\;\triangle_{123},
    \diamondsuit^2_{135},\diamondsuit^3_{125},
    \Theta_{23},\Lambda^2_{13},\Lambda^2_{35},\Lambda^3_{12}, \Lambda^3_{25}\}$;
\item $(\bar\rho)^{-1}(\triangle_{245})=\{\triangle_{245},\;\triangle_{124},\diamondsuit^2_{145},
\diamondsuit^4_{125},\Theta_{24},\Lambda^2_{14},\Lambda^2_{45}, \Lambda^4_{12},\Lambda^4_{25}\}$;
\item $(\bar\rho)^{-1}(\triangle_{345})=\{\triangle_{345},\;\triangle_{134},\diamondsuit^3_{145},\diamondsuit^4_{135},\Theta_{34},\Lambda^3_{14},\Lambda^3_{45},\Lambda^4_{13}, \Lambda^4_{35}\}$;
\item $(\bar\rho)^{-1}(\diamondsuit^2_{345})=\{\diamondsuit^2_{345},\;\diamondsuit^2_{134}, \Lambda^2, \Lambda^2_{34}\}$;
\item $(\bar\rho)^{-1}(\diamondsuit^3_{245})=\{\diamondsuit^3_{245},\;\diamondsuit^3_{124},\Lambda^3, \Lambda^3_{24}\}$;
\item $(\bar\rho)^{-1}(\diamondsuit^4_{235})=\{\diamondsuit^4_{235},\;\diamondsuit^4_{123},\Lambda^4,
    \Lambda^4_{23}\}$;
\item $(\bar\rho)^{-1}(\diamondsuit^5_{234})=\{\diamondsuit^5_{234},\;\diamondsuit^1_{234}, \Lambda^1_{23},\Lambda^1_{24},\Lambda^1_{34},\Lambda^2_{15}, \Lambda^3_{15},\Lambda^4_{15},\Lambda^5_{23},\Lambda^5_{24},\Lambda^5_{34},
\Theta_{12},\Theta_{13},\Theta_{14},\Theta_{25},\Theta_{35},
\Theta_{45},\bigcirc\}$.
\end{itemize}

By Theorem~\ref{t:r=2}, for any automorphism $\psi$ of $\lambda(M_{2,4})$ the restriction $\psi|\lambda(C_4)$ is
an automorphism of $\lambda(C_4)$ such that $\psi(x)=x$ for all $x\in C_4$. The description of
the automorphism group of $\lambda(C_4)$ given in~\cite{BG9} ensures that $\psi(\diamondsuit^{n}_{ijk})=\diamondsuit^n_{ijk}$
for any $\diamondsuit^n_{ijk}\in\lambda(C_4)$. Taking into account that
$|(\bar\rho)^{-1}(\triangle_{234})|=1<|(\bar\rho)^{-1}(\triangle_{ijk})|$ for
any $\triangle_{ijk}\in\lambda(C_4)\setminus\{\triangle_{234}\}$, we conclude that $\psi(\triangle_{234})=\psi(\triangle_{234})$
and hence $\psi(x\triangle_{234})=x\triangle_{234}$ for any $x\in C_4$. Since $\{x\triangle_{234}:x\in C_4\}=\{\triangle_{ijk}:2\le i<j<k\le 5\}$,
we conclude that $\psi|\lambda(C_4)$ is the identity automorphism of the semigroup $\lambda(C_4)$.

By Theorem~\ref{t:r=2}, the automorphism group $\Aut(\lambda(M_{2,4}))$ is isomorphic
to $$\prod_{\mathcal L\in\lambda(C_4)}S_{(\bar\rho)^{-1}(\mathcal L)\setminus\{\mathcal L\}}\cong S_3^3\times S_8^3\times S_{17}\times S_{19}.$$

\subsection{The semigroup $\lambda(\M_{3,3})$}

In this section we recognize the automorphism group of the superextension of the semigroup $\M_{3,3}=\{a,a^2,a^3,a^4,a^5\}$.
This semigroup is generated by an element $a$ such that $a^6=a^3$. Observe that the ideal $\M_{3,3}^{\cdot2}=\{a^2,a^3,a^4,a^5\}$
is isomorphic to the monogenic semigroup $\M_{2,3}=\{b,b^2,b^3,b^4\}$ as for the element $b=a^2$ we get $\{b,b^2,b^3,b^4\}=\{a^2,a^4,a^3,a^5\}$.

As shown in Subsection~\ref{ss:M23}, the automorphism group $\Aut(\lambda(\M_{2,3}))$ consists of all bijections
$\psi$ of $\lambda(\M_{2,3})$ such that
\begin{itemize}
\item $\psi(x)=x$ for any $x\notin\{\triangle_4,\square_2,\square_3\}\cup\{b,\triangle_2,\triangle_3,\square_1,\square_4\}$;
\item the sets $\{\triangle_4,\square_2,\square_3\}$ and $\{b,\triangle_2,\triangle_3,\square_1,\square_4\}$ are $\psi$-invariant.
 \end{itemize}
Taking into account the isomorphism $\{b,b^2,b^3,b^4\}\to\{a^2,a^4,a^3,a^5\}$ between the semigroups $\M_{2,3}$ and $\M_{3,3}^{\cdot2}$, we obtain the following fact.

\begin{claim}\label{cl:M33} The automorphism group of the semigroup $\lambda(\M_{3,3}^{\cdot2})$ consists of bijections $\psi$ of $\lambda(\M_{3,3}^{\cdot2})$ such that
\begin{itemize}
\item $\psi(x)=x$ for any $x\not\in\{\triangle_{234},\diamondsuit^4_{235},\diamondsuit^3_{245}\}\cup \{a^2,\triangle_{235},\triangle_{245},\diamondsuit^2_{345},\diamondsuit^5_{234}\}$;
\item the sets $\{\triangle_{234},\diamondsuit^4_{235},\diamondsuit^3_{245}\}$ and
$\{a^2,\triangle_{235},\triangle_{245},\diamondsuit^2_{345},\diamondsuit^5_{234}\}$ are $\psi$-invariant.
\end{itemize}
\end{claim}
By Lemma~\ref{l:auto-shift}, the map $\bar\sigma:\lambda(\M_{3,3})\to\lambda(\M_{3,3}^{\cdot2})$, $\bar \sigma:\A\mapsto a*\A$,
is an auto-shift of the semigroup $\lambda(\M_{3,3})$.
By  routine calculations, it can be shown that the map $\bar\sigma$ has the following fibers:
\begin{itemize}
\item $\bar\sigma^{-1}(a^2)=\{a\}$, $\bar\sigma^{-1}(a^4)=\{a^3\}$, $\bar\sigma^{-1}(a^5)=\{a^4\}$;
\item $\bar\sigma^{-1}(a^3)=\{a^2,a^5,\triangle_{125},\triangle_{235},\triangle_{245},\Lambda^2_{15},\Lambda^2_{35},
\Lambda^2_{45},\Lambda^5_{12},\Lambda^5_{23},\Lambda^5_{24},\diamondsuit^2_{135}, \diamondsuit^2_{145}, \diamondsuit^2_{345},
\diamondsuit^5_{123},\diamondsuit^5_{124}, \diamondsuit^5_{234},\Theta_{25},\Lambda^2,\Lambda^5\}$;
\item $\bar\sigma^{-1}(\triangle_{245})=\{\triangle_{134}\}$;
\item $\bar\sigma^{-1}(\triangle_{234})=\{\triangle_{123},\triangle_{135},\Lambda^1_{23},\Lambda^1_{35},\Lambda^3_{12},\Lambda^3_{15},
\diamondsuit^1_{235},\diamondsuit^3_{125}, \Theta_{13}\}$;
\item $\bar\sigma^{-1}(\triangle_{235})=\{\triangle_{124},\triangle_{145}, \Lambda^1_{24}, \Lambda^1_{45}, \Lambda^4_{12},\Lambda^4_{15},  \diamondsuit^1_{245},
\diamondsuit^4_{125}, \Theta_{14}\}$;
\item $\bar\sigma^{-1}(\triangle_{345})=\{\triangle_{234},\triangle_{345},\Lambda^3_{24},\Lambda^3_{45},\Lambda^4_{23},
\Lambda^4_{35},\diamondsuit^3_{245},\diamondsuit^4_{235},\Theta_{34}\}$;
\item $\bar\sigma^{-1}(\diamondsuit^2_{345})= \{\Lambda^1_{34}, \diamondsuit^1_{234}, \diamondsuit^1_{345},\Lambda^1\}$;
\item $\bar\sigma^{-1}(\diamondsuit^4_{235})=\{\Lambda^3_{14},\diamondsuit^3_{124}, \diamondsuit^3_{145},\Lambda^3\}$;
\item $\bar\sigma^{-1}(\diamondsuit^5_{234})=\{\Lambda^4_{13},\diamondsuit^4_{123}, \diamondsuit^4_{135},\Lambda^4\}$;
\item $\bar\sigma^{-1}(\diamondsuit^3_{245})=\{\Lambda^1_{25},\Lambda^2_{13},\Lambda^2_{14},\Lambda^2_{34},\Lambda^3_{25},
\Lambda^4_{25},\Lambda^5_{13}, \Lambda^5_{14}, \Lambda^5_{34},  \diamondsuit^5_{134}, \diamondsuit^2_{134}, \Theta_{12},\Theta_{15},\Theta_{23}, \Theta_{24},
\Theta_{35}, \Theta_{45}, \bigcirc\}$.
\end{itemize}

Then
\begin{itemize}
\item $\bar\sigma^{-1}(x)\setminus\lambda(\M_{3,3}^{\cdot2})=\emptyset$ for $x\in\{a^2,a^4,a^5\}$;
\item $\bar\sigma^{-1}(a^3)\setminus\lambda(\M_{3,3}^{\cdot2})=
    \{\triangle_{125},\Lambda^2_{15},\Lambda^2_{35}, \Lambda^2_{45},\Lambda^5_{12},\Lambda^5_{23},\Lambda^5_{24},\diamondsuit^2_{135},
    \diamondsuit^2_{145},\diamondsuit^5_{123},\diamondsuit^5_{124},\Theta_{25},\Lambda^2,\Lambda^5\}$;
\item $\bar\sigma^{-1}(\triangle_{245})\setminus\lambda(\M_{3,3}^{\cdot2})=\{\triangle_{134}\}$;
\item $\bar\sigma^{-1}(\triangle_{234})\setminus\lambda(\M_{3,3}^{\cdot2})=\{\triangle_{123},\triangle_{135},\Lambda^1_{23},\Lambda^1_{35},\Lambda^3_{12},\Lambda^3_{15},
\diamondsuit^1_{235},\diamondsuit^3_{125}, \Theta_{13}\}$;
\item $\bar\sigma^{-1}(\triangle_{235})\setminus\lambda(\M_{3,3}^{\cdot2})=\{\triangle_{124},\triangle_{145}, \Lambda^1_{24}, \Lambda^1_{45}, \Lambda^4_{12},\Lambda^4_{15}, \diamondsuit^1_{245},
\diamondsuit^4_{125}, \Theta_{14}\}$;
\item $\bar\sigma^{-1}(\triangle_{345})\setminus\lambda(\M_{3,3}^{\cdot2})=\{\Lambda^3_{24},\Lambda^3_{45},\Lambda^4_{23}, \Lambda^4_{35},\Theta_{34}\}$;
\item $\bar\sigma^{-1}(\diamondsuit^2_{345})\setminus\lambda(\M_{3,3}^{\cdot2})= \{\Lambda^1_{34}, \diamondsuit^1_{234},
\diamondsuit^1_{345},\Lambda^1\}$;
\item $\bar\sigma^{-1}(\diamondsuit^4_{235})\setminus\lambda(\M_{3,3}^{\cdot2})=\{\Lambda^3_{14},\diamondsuit^3_{124}, \diamondsuit^3_{145},\Lambda^3\}$;
\item $\bar\sigma^{-1}(\diamondsuit^5_{234})\setminus\lambda(\M_{3,3}^{\cdot2})=\{\Lambda^4_{13},\diamondsuit^4_{123}, \diamondsuit^4_{135},\Lambda^4\}$;

\item $\bar\sigma^{-1}(\diamondsuit^3_{245})\setminus\lambda(\M_{3,3}^{\cdot2})=
    \{\Lambda^1_{25},\Lambda^2_{13},\Lambda^2_{14},\Lambda^2_{34},\Lambda^3_{25}, \Lambda^4_{25},\Lambda^5_{13}, \Lambda^5_{14}, \Lambda^5_{34},
    \diamondsuit^5_{134}, \diamondsuit^2_{134}, \Theta_{12},\Theta_{15},\Theta_{23}, \Theta_{24}, \Theta_{35}, \Theta_{45}, \bigcirc\}$.
\end{itemize}

By Corollary~\ref{c:Ker}, for any automorphism $\psi$ of $\lambda(\M_{3,3})$ we get $|\bar\sigma^{-1}(\psi(x))|=|\bar\sigma^{-1}(x)|$ for all
$x\in\lambda(\M_{3,3}^{\cdot 2})$. Taking into account Corollary~\ref{c:Ker}, Claim~\ref{cl:M33} and comparing the cardinalities of
the fibers of $\bar\sigma$, we conclude that $\psi$ is the identity on the set $\lambda(\M_{3,3}^{\cdot2})\setminus\{\diamondsuit^2_{345},\diamondsuit^5_{234}\}$.

To see that $\psi|\lambda(\M_{3,3}^{\cdot2})$ is trivial, it is sufficient to show that the assumption $\psi(\diamondsuit^2_{345})=\diamondsuit^5_{234}$ leads to a contradiction. In this case $\psi(\Lambda^1)\in \{\Lambda^4_{13},\diamondsuit^4_{123}, \diamondsuit^4_{135},\Lambda^4\}$.
Replacing $\psi$ by the composition of $\psi$ with a suitable permutation in the kernel of the restriction
operator $R:\Aut(\lambda(\M_{3,3}))\to\Aut(\lambda(\M_{3,3}^{\cdot2}))$,
we can assume that $\psi(\Lambda^1)=\Lambda^4$ and $\psi(\triangle_{123})=\triangle_{123}$.

Now observe that $\triangle_{123}*\Lambda^1=\triangle_{234}\ne \triangle_{345}=\triangle_{123}*\Lambda^4.$
 On the other hand,
$$\triangle_{234}=\psi(\triangle_{234})=\psi(\triangle_{123}*\Lambda^1)=\psi(\triangle_{123})*\psi(\Lambda^1)=
\triangle_{123}*\Lambda^4=\triangle_{345},$$which is a desired
contradiction showing that the restriction $\psi|\lambda(\M_{3,3}^{\cdot2})$ is trivial.

Therefore, the restriction operator $R$ has trivial range and the automorphism
group $\Aut(\lambda(\M_{3,3}))$ coincides with the kernel of $R$, which is isomorphic to
$$\prod_{\mathcal L\in\lambda(\M_{3,3}^{\cdot2})}S_{\bar\sigma^{-1}(\mathcal L)\setminus\lambda(\M_{3,3}^{\cdot2})}\cong
S_4^3\times S_{5}\times S_9^2\times S_{14}\times S_{18}.$$

\subsection{The semigroup $\lambda(\M_{4,2})$}

In this section we recognize the structure of the automorphism group of the superextension of the semigroup $\M_{4,2}$ and its
characteristic ideals $\M_{4,2}^{\cdot k}$ for $k\in\{2,3,4\}$. The monogenic semigroup $\M_{4,2}=\{a,a^2,a^3,a^4,a^5\}$
is generated by an element $a$ such that $a^6=a^4$.
The characteristic ideal $\M_{4,2}^{\cdot4}=\{a^4,a^5\}$ is a group with neutral element $a^4$.

Observe that $\lambda(\M_{4,2}^{\cdot 4})=\M_{4,2}^{\cdot4}\cong C_2$ and hence $\Aut(\lambda(\M_{4,2}^{\cdot4}))\cong\Aut(C_2)\cong C_1$.

To recognize the structure of the automorphism groups of the superextensions of the semigroups $\M_{4,2}^{\cdot3}=\{a^3,a^4,a^5\}$
and $\M_{4,2}^{\cdot2}=\{a^2,a^3,a^4,a^5\}$, observe that the map $\rho:\M_{4,2}^{\cdot 2}\to \M_{4,2}^{\cdot4}$, $\rho:x\mapsto a^4x=a^2x$,
is a homomorphic retraction of $\M_{4,2}^{\cdot 2}$ onto the group $\M_{4,2}^{\cdot4}$ such that $xy=\rho(x)\cdot\rho(y)$ for
all $x,y\in\M_{4,2}^{\cdot2}$. This homomorphic retraction induces a homomorphic retraction
$$\bar \rho:\lambda(\M_{4,2}^{\cdot 2})\to\lambda(\M_{4,2}^{\cdot4}),\;\;\bar\rho:\A\mapsto a^4*\A,$$
such that $\A*\mathcal B=\bar\rho(\A)*\bar\rho(\mathcal B)$ for any $\A,\mathcal B\in\lambda(\M_{4,2}^{\cdot2})$.

Using the notations for the elements of the superextension $\lambda(\M_{4,2})$ from Subsection~\ref{ss:M24}, observe that $$\lambda(\M_{4,2}^{\cdot2})=\{a^i:2\le i\le 5\}\cup\{\triangle_{ijk}:2\le i<j<k\le 5\}\cup\big\{\diamondsuit^n_{ijk}:2\le i<j<k\le 5,\;n\in\{2,3,4,5\}\setminus\{i,j,k\}\big\}.$$
By routine calculations it can be shown that the map $\bar\rho:\lambda(\M_{4,2}^{\cdot2})\to\lambda(\M_{4,2}^{\cdot4})=\{a^4,a^5\}$ has the following fibers:
\begin{itemize}
\item $\bar\rho^{-1}(a^4)=\{a^2,a^4,\triangle_{234},\triangle_{245}, \diamondsuit^2_{345},\diamondsuit^4_{235}\}$;
\item $\bar\rho^{-1}(a^5)=\{a^3,a^5,\triangle_{235},\triangle_{345}, \diamondsuit^3_{245},\diamondsuit^5_{234}\}$.
\end{itemize}
Now we see that $\Aut(\lambda(\M_{4,2}^{\cdot3}))$ is isomorphic to $$\prod_{i\in\{4,5\}}S_{\bar\rho^{-1}(a^i)\cap\lambda(\M_{4,2}^{\cdot3}) \setminus\M_{4,2}^{\cdot4}}=S_{\{a^3,\triangle_{345}\}}\cong C_2$$
and $\Aut(\lambda(\M_{4,2}^{\cdot2}))$ is isomorphic to
$$\prod_{i\in\{4,5\}}S_{\bar\rho^{-1}(a^i)\setminus\M_{4,2}^{\cdot4}}\cong
S_5\times S_5.$$

To detect the algebraic structure of the automorphism group of $\lambda(\M_{4,2})$, consider the map $\sigma:\M_{4,2}\to\M_{4,2}^{\cdot2}$, $\sigma:x\mapsto ax$, which induces the map
$$\bar\sigma:\lambda(\M_{4,2})\to\lambda(\M_{4,2}^{\cdot2}),\;\;\bar\sigma:\A\mapsto a*\A.$$
By Lemma~\ref{l:auto-shift}, $\bar\sigma$ is an auto-shift of $\lambda(\M_{4,2})$.

By routine calculations we can show that the map $\bar\sigma$ has the following fibers:
\begin{itemize}
\item $\bar\sigma^{-1}(a^2)=\{a\}$, $\bar\sigma^{-1}(a^3)=\{a^2\}$, $\bar\sigma^{-1}(a^5)=\{a^4\}$;
\item $\bar\sigma^{-1}(a^4)=\{a^3,a^5,\triangle_{135},\triangle_{235}, \triangle_{345},\Lambda^3_{15},\Lambda^3_{25}, \Lambda^3_{45},
\Lambda^5_{13}, \Lambda^5_{23}, \Lambda^5_{34},\diamondsuit^3_{125}, \diamondsuit^3_{145}, \diamondsuit^3_{245},\diamondsuit^5_{123},
\diamondsuit^5_{134}, \diamondsuit^5_{234}, \Lambda^3,\Lambda^5, \Theta_{35}\}$;
\item $\bar\sigma^{-1}(\triangle_{235})=\{\triangle_{124}\}$;
\item $\bar\sigma^{-1}(\triangle_{234})=\{\triangle_{123},\triangle_{125}, \Lambda^1_{23}, \Lambda^1_{25}, \Lambda^2_{13},\Lambda^2_{15}, \diamondsuit^1_{235}, \diamondsuit^2_{135}, \Theta_{12}\}$;
\item $\bar\sigma^{-1}(\triangle_{245})=\{\triangle_{134},\triangle_{145}, \Lambda^1_{34},\Lambda^1_{45},
\Lambda^4_{13},\Lambda^4_{15}, \diamondsuit^1_{345}, \diamondsuit^4_{135},\Theta_{14}\}$;
\item $\bar\sigma^{-1}(\triangle_{345})=\{\triangle_{234},\triangle_{245}, \Lambda^2_{34}, \Lambda^2_{45}, \Lambda^4_{23}, \Lambda^4_{25},
\diamondsuit^2_{345}, \diamondsuit^4_{235},\Theta_{24}\}$;
\item $\bar\sigma^{-1}(\diamondsuit^2_{345}) = \{\diamondsuit^1_{234},\diamondsuit^1_{245},\Lambda^1, \Lambda^1_{24}\}$;
\item $\bar\sigma^{-1}(\diamondsuit^3_{245})= \{\diamondsuit^2_{134}, \diamondsuit^2_{145}, \Lambda^2, \Lambda^2_{14}\}$;
\item $\bar\sigma^{-1}(\diamondsuit^5_{234})= \{\diamondsuit^4_{123},\diamondsuit^4_{125},\Lambda^4, \Lambda^4_{12}\}$;
\item $\bar\sigma^{-1}(\diamondsuit^4_{235})= \{\Lambda^1_{35},\Lambda^2_{35},\Lambda^3_{12},\Lambda^3_{14},
\Lambda^3_{24},\Lambda^4_{35},\Lambda^5_{12},\Lambda^5_{14},\Lambda^5_{24}, \diamondsuit^3_{124},\diamondsuit^5_{124},
\Theta_{13},\Theta_{15},\Theta_{23}, \Theta_{25},\Theta_{34},\Theta_{45},\bigcirc\}$.
\end{itemize}

Now consider the restriction operator $$R:\Aut(\lambda(\M_{4,2}))\to\Aut(\lambda(\M_{4,2}^{\cdot2})),\;\;R:\psi\mapsto\psi|\lambda(\M_{4,2}^{\cdot2}).$$
 We claim that this operator has trivial range.

Fix any automorphism $\psi\in\Aut(\lambda(\M_{4,2}))$. Proposition~\ref{mono1} ensures that $\psi(a^i)=a^i$ for all $i\in\{1,2,3,4,5\}$.

The description of the automorphism group $\Aut(\lambda(\M_{4,2}^{\cdot2}))$ ensures that the sets $\{a^2,a^4,\triangle_{234},\triangle_{245}, \diamondsuit^2_{345},\diamondsuit^4_{235}\}$ and $\{a^3,a^5,\triangle_{235},\triangle_{345}, \diamondsuit^3_{245},\diamondsuit^5_{234}\}$ are $\psi$-invariant.
Then the sets $F_4=\{\triangle_{234},\triangle_{245}, \diamondsuit^2_{345},\diamondsuit^4_{235}\}$
and $F_4'=\{\triangle_{235},\triangle_{345}, \diamondsuit^3_{245},\diamondsuit^5_{234}\}$ are $\psi$-invariant, too.
Taking into account that the ideal $\lambda(\M_{4,2}^{\cdot3})=\{a^3,a^4,a^5,\triangle_{345}\}$ is characteristic,
we conclude that $\psi(\triangle_{345})=\triangle_{345}$. Consequently, the set
$F_3=\{\triangle_{235}, \diamondsuit^3_{245},\diamondsuit^5_{234}\}$ is $\psi$-invariant and $\psi(\A)=\A$ for any
$\A\in\lambda(\M_{4,2}^{\cdot2})\setminus(F_3\cup F_4)$.

It follows from $\psi\circ\bar\sigma=\bar\sigma\circ\psi$ that $$|\bar\sigma^{-1}(\psi^{-1}(\A))\setminus\lambda(\M_{4,2}^{\cdot2})|=
|\bar\sigma^{-1}(\A)\setminus\lambda(\M_{4,2}^{\cdot2})|$$ for any $\A\in\lambda(\M_{4,2}^{\cdot2})$.
Observe that
\begin{itemize}
\item $|\bar\sigma^{-1}(x)\setminus\lambda(\M_{4,2}^{\cdot2})|=0$ for $x\in\{a^2,a^3,a^5\}$;

\item $|\bar\sigma^{-1}(a^4)\setminus\lambda(\M_{4,2}^{\cdot2})|= |\{\triangle_{135},\Lambda^3_{15},\Lambda^3_{25}, \Lambda^3_{45},\Lambda^5_{13},
\Lambda^5_{23}, \Lambda^5_{34},\diamondsuit^3_{125}, \diamondsuit^3_{145}, \diamondsuit^5_{123},\diamondsuit^5_{134}, \Lambda^3,\Lambda^5, \Theta_{35}\}|=14$;
\item $|\bar\sigma^{-1}(\triangle_{235})\setminus\lambda(\M_{4,2}^{\cdot2})|=|\{\triangle_{124}\}|=1$;
\item $|\bar\sigma^{-1}(\triangle_{234})\setminus\lambda(\M_{4,2}^{\cdot2})|=
   |\{\triangle_{123},\triangle_{125}, \Lambda^1_{23}, \Lambda^1_{25}, \Lambda^2_{13},\Lambda^2_{15}, \diamondsuit^1_{235}, \diamondsuit^2_{135}, \Theta_{12}\}|=9$;
\item  $|\bar\sigma^{-1}(\triangle_{245})\setminus\lambda(\M_{4,2}^{\cdot2})|=
|\{\triangle_{134},\triangle_{145}, \Lambda^1_{34},\Lambda^1_{45}, \Lambda^4_{13},\Lambda^4_{15},
\diamondsuit^1_{345},\diamondsuit^4_{135},\Theta_{14}\}|=9$;
\item  $|\bar\sigma^{-1}(\triangle_{345})\setminus\lambda(\M_{4,2}^{\cdot2})|=
|\{\Lambda^2_{34}, \Lambda^2_{45}, \Lambda^4_{23}, \Lambda^4_{25},  \Theta_{24}\}|=5$;
\item $|\bar\sigma^{-1}(\diamondsuit^2_{345})\setminus\lambda(\M_{4,2}^{\cdot2})|=
|\{\diamondsuit^1_{234},\diamondsuit^1_{245},\Lambda^1,\Lambda^1_{24}\}|=4$;
\item $|\bar\sigma^{-1}(\diamondsuit^3_{245})\setminus\lambda(\M_{4,2}^{\cdot2})|=
|\{\diamondsuit^2_{134},\diamondsuit^2_{145}, \Lambda^2,  \Lambda^2_{14}\}|=4$;
\item $|\bar\sigma^{-1}(\diamondsuit^5_{234})\setminus\lambda(\M_{4,2}^{\cdot2})|=
|\{\diamondsuit^4_{123},\diamondsuit^4_{125},\Lambda^4, \Lambda^4_{12}\}|=4$;
\item  $|\bar\sigma^{-1}(\diamondsuit^4_{235})\setminus\lambda(\M_{4,2}^{\cdot2})|=$\newline $=|\{\Lambda^1_{35},\Lambda^2_{35},\Lambda^3_{12},
\Lambda^3_{14}, \Lambda^3_{24},\Lambda^4_{35},\Lambda^5_{12},\Lambda^5_{14},\Lambda^5_{24}, \diamondsuit^3_{124},\diamondsuit^5_{124},
\Theta_{13},\Theta_{15},\Theta_{23}, \Theta_{25},\Theta_{34},\Theta_{45},\bigcirc\}|=18$.

\end{itemize}
Comparing the cardinalities of the sets $\bar\sigma^{-1}(\A)\setminus\lambda(\M_{4,2}^{\cdot2})$ for points in the sets $F_3$ and $F_4$,
we can conclude that the sets $\{\triangle_{234},\triangle_{245}\}$ and $\{\diamondsuit^3_{245},\diamondsuit^5_{234}\}$ are $\psi$-invariant and $\psi(\A)=\A$ for any $\A\in \lambda(\M_{4,2}^{\cdot2})\setminus
\{\triangle_{234},\triangle_{245},\diamondsuit^3_{245},\diamondsuit^5_{234}\}$.

Now we show that $\psi(\triangle_{234})=\triangle_{234}$. In the opposite case, $\psi(\triangle_{234})=\triangle_{245}$ and hence $\psi(\triangle_{123})\in\bar\sigma^{-1}(\triangle_{245})\setminus\lambda(\M_{4,2}^{\cdot2})$.
Replacing $\psi$ by the composition of $\psi$ with a suitable permutation in the kernel of the restriction operator $R:\Aut(\lambda(\M_{4,2}))\to\Aut(\lambda(\M_{4,2}^{\cdot2}))$, we can assume that $\psi(\triangle_{123})=\triangle_{134}$.

Taking into account that $\triangle_{123}*\triangle_{123}=\diamondsuit^4_{235}$ and $\triangle_{134}*\triangle_{134}=\triangle_{245}$, we conclude that
$$\diamondsuit^4_{235}=\psi(\diamondsuit^4_{235})=\psi(\triangle_{123})*
\psi(\triangle_{123})=\triangle_{134}*\triangle_{134}=\triangle_{245},$$which is a contradiction proving that $\psi(\A)=\A$ for all $\A\in \lambda(\M_{4,2}^{\cdot2})\setminus
\{\diamondsuit^3_{245},\diamondsuit^5_{234}\}$.

Assuming that $\psi|\lambda(\M_{4,2}^{\cdot2})$ is not identity, we conclude that $\psi(\diamondsuit^3_{245})=\diamondsuit^5_{234}$. Replacing $\psi$ by the composition of $\psi$ with a suitable permutation in the kernel of the restriction operator $R$, we can assume that $\psi(\Lambda^2)=\Lambda^4$ and $\psi(\triangle_{123})=\triangle_{123}$.

Taking into account that $\triangle_{123}*\Lambda^2=\triangle_{345}$ and $\triangle_{123}*\Lambda^4=\diamondsuit^5_{234}$, we conclude that
$$\triangle_{345}=\psi(\triangle_{345})=\psi(\triangle_{123}*\Lambda^2)= \psi(\triangle_{123})*\psi(\Lambda^2)=\triangle_{123}*\Lambda^4=\diamondsuit^5_{234},$$
which is a contradiction completing the proof of the triviality of the range of the restriction operator $R:\Aut(\lambda(\M_{4,2}))\to\Aut(\lambda(\M_{4,2}^{\cdot2}))$.

Then the automorphism group
$\Aut(\lambda(\M_{4,2}))$ coincides with the kernel of $R$, which is isomorphic to the group
$$\prod_{\mathcal L\in\lambda(\M_{4,2}^{\cdot2})}S_{\bar\sigma^{-1}(\mathcal L)\setminus\lambda(\M_{4,2}^{\cdot2})}\cong
S_4^3\times S_5\times S_9^2\times S_{14}\times S_{18}.$$

\subsection{The semigroup $\lambda(\M_{5,1})$}

In this section we recognize the structure of the automorphism group of the superextension of the semigroup $\M_{5,1}$ and its
characteristic ideals $\M_{5,1}^{\cdot k}$ for $k\in\{2,3,4,5\}$. The monogenic semigroup $\M_{5,1}=\{a,a^2,a^3,a^4,a^5\}$ is
generated by an element $a$ such that $a^6=a^5$.
Observe that for any $x, y\in\{a^3,a^4,a^5\}=\M_{5,1}^{\cdot3}$ we get $xy=a^5$. This implies that for any $k\in\{3,4,5\}$ any
bijection $\psi$ of $\lambda(\M_{5,1}^{\cdot k})$ with $\psi(a^5)=a^5$ is an automorphism of the semigroup $\lambda(\M_{5,1}^{\cdot k})$.
This observation implies that $\Aut(\lambda(\M_{5,1}^{\cdot5}))\cong \Aut(\lambda(\M_{5,1}^{\cdot4}))\cong C_1$ and $\Aut(\lambda(\M_{5,1}^{\cdot 3}))\cong S_3$.

Next, consider the semigroup $\M_{5,1}^{\cdot2}=\{a^2,a^3,a^4,a^5\}$ and its superextension $\lambda(\M_{5,1}^{\cdot2})$. It follows from $\M_{5,1}^{\cdot2}*\M_{5,1}^{\cdot2}=\M_{5,1}^{\cdot 4}=\{a^4,a^5\}$ that $\lambda(\M_{5,1}^{\cdot2})*\lambda(\M_{5,1}^{\cdot2})=\lambda(\M_{5,1}^{\cdot4})=\M_{5,1}^{\cdot4}$.
This implies that the ideal $\M_{5,1}^{\cdot4}$ is characteristic in the semigroup $\lambda(\M_{5,1}^{\cdot2})$. Since the semigroup $\M_{5,1}^{\cdot4}=\{a^4,a^5\}$ has trivial automorphism group, $\psi(a^4)=a^4$ and $\psi(a^5)=a^5$ for any automorphism $\psi$ of $\lambda(\M_{5,1}^{\cdot2})$.

It is easy to show that the equation $x*y=a^4$ has a unique solution $x=y=a^2$ in the semigroup $\lambda(\M_{5,1}^{\cdot2})$.
Consequently, $\psi(a^2)=a^2$ for any automorphism $\psi$ of $\lambda(\M_{5,1}^{\cdot2})$ and $\A*\mathcal B=a^5$ for any maximal
linked upfamilies $\A,\mathcal B\in\lambda(\M_{5,1}^{\cdot2})$ such that $(\A,\mathcal B)\ne(a^2,a^2)$. It follows that any bijection
$\psi$ of $\lambda(\M_{5,1}^{\cdot2})$ such that $\psi(a^i)=a^i$ for $i\in\{2,4,5\}$ is an automorphism of the semigroup
$\lambda(\M_{5,1}^{\cdot 2})$. This implies that $\Aut(\M_{5,1}^{\cdot2})\cong S_9$.

To recognize the automorphism group of $\lambda(\M_{5,1})$, consider
the map $$\bar\sigma:\lambda(\M_{5,1})\to\lambda(\M_{5,1}^{\cdot2}),\;\;\bar\sigma:\A\mapsto a*\A.$$
By Lemma~\ref{l:auto-shift}, $\bar\sigma$ is an auto-shift of $\lambda(\M_{5,1})$ and its restriction
$\bar\sigma_2:=\bar\sigma|\lambda(\M_{5,1}^{\cdot2})$ is a good shift of $\lambda(\M_{5,1}^{\cdot2})$.

By routine calculations, it can be shown that the map $\bar\sigma$ has the following fibers:
\begin{itemize}
\item $\bar\sigma^{-1}(a^{i})=\{a^{i-1}\}$ for $i\in\{2,3,4\}$;
\item $\bar\sigma^{-1}(a^5)=\{a^4,a^5,\Lambda^5,\Lambda^4,\diamondsuit^5_{234}, \diamondsuit^5_{134},\diamondsuit^5_{124},\diamondsuit^4_{235},
\diamondsuit^4_{135}, \diamondsuit^4_{125},\Lambda^5_{34},\Lambda^5_{24},\Lambda^5_{14},\Lambda^4_{35}, \Lambda^4_{25},\Lambda^4_{15},
\triangle_{345},\triangle_{245}, \triangle_{145}, \Theta_{45}\}$;
\item $\bar\sigma^{-1}(\triangle_{234})=\{\triangle_{123}\}$;
\item $\bar\sigma^{-1}(\triangle_{235})=\{\Theta_{12},\diamondsuit^2_{145}, \diamondsuit^1_{245},\Lambda^2_{15},\Lambda^2_{14},\Lambda^1_{25},\Lambda^1_{24}, \triangle_{125},\triangle_{124}\}$;
\item $\bar\sigma^{-1}(\triangle_{245})=\{\Theta_{13},\diamondsuit^3_{145}, \diamondsuit^1_{345},\Lambda^3_{15},\Lambda^3_{14},\Lambda^1_{35},\Lambda^1_{34}, \triangle_{135},\triangle_{134}\}$;
\item $\bar\sigma^{-1}(\triangle_{345})=\{\Theta_{23},\diamondsuit^3_{245}, \diamondsuit^2_{345},\Lambda^3_{25},\Lambda^3_{24},\Lambda^2_{35},\Lambda^2_{34}, \triangle_{234},\triangle_{235}\}$;
\item $\bar\sigma^{-1}(\diamondsuit^2_{345})=\{\Lambda^1,\diamondsuit^1_{235}, \diamondsuit^1_{234},\Lambda^1_{23}\}$;
\item $\bar\sigma^{-1}(\diamondsuit^3_{245})=\{\Lambda^2,\diamondsuit^2_{135}, \diamondsuit^2_{134},\Lambda^2_{13}\}$;
\item $\bar\sigma^{-1}(\diamondsuit^4_{235})=\{\Lambda^3,\diamondsuit^3_{125}, \diamondsuit^3_{124},\Lambda^3_{12}\}$;
\item $\bar\sigma^{-1}(\diamondsuit^5_{234})=\{\bigcirc,\Theta_{35}, \Theta_{34},\Theta_{25},\Theta_{24},\Theta_{15},\Theta_{14},
\diamondsuit^5_{123}, \diamondsuit^4_{123},\Lambda^5_{23},\Lambda^5_{13},\Lambda^5_{12},\Lambda^4_{23},\Lambda^4_{13}, \Lambda^4_{12},\Lambda^3_{45},
\Lambda^2_{45}, \Lambda^1_{45}\}$.
\end{itemize}

To recognize the algebraic structure of the automorphism group $\Aut(\lambda(\M_{5,1}))$, consider the restriction operator
$R:\Aut(\lambda(\M_{5,1}))\to\Aut(\lambda(\M_{5,1}^{\cdot2}))$. By Corollary~\ref{c:Ker}, the kernel of this operator is isomorphic to
$$\prod_{\mathcal L\in\lambda(\M_{5,1}^{\cdot2})}S_{\bar\sigma^{-1}(\mathcal L)\setminus \lambda(\M_{5,1}^{\cdot2})}\cong
S_{14}\times S_9^2\times S_5\times S_4^3\times S_{18}.$$

We claim that the operator $R$ has trivial range. Given any automorphism $\psi$ of $\lambda(\M_{5,1})$, we should prove that $\psi(\A)=\A$ for any $\A\in\lambda(\M_{5,1}^{\cdot2})$. By Proposition~\ref{mono1}, $\psi(a^i)=a^i$ for all $i\in\{1,2,3,4,5\}$. The equality $\psi(a)=a$ implies that $\bar\sigma\circ\psi=\psi\circ\bar\sigma$. Consequently, for every $\mathcal L\in \lambda(\M_{5,1}^{\cdot2})$ we have $\psi(\bar\sigma^{-1}(\mathcal L))=\bar\sigma^{-1}(\psi(\mathcal L))$. Since the ideals $\lambda(\M_{5,1}^{\cdot k})$ are characteristic in $\lambda(\M_{5,1})$, Corollary~\ref{c:Ker} ensures that
   $$|\bar\sigma^{-1}(\psi(\mathcal L))\cap\lambda(\M_{5,1}^{\cdot k})|=|\bar\sigma^{-1}(\mathcal L)\cap\lambda(\M_{5,1}^{\cdot k})|$$ for all
$\mathcal L\in\lambda(\M_{5,1}^{\cdot 2})$ and all $k\in\{1,2,3,4,5\}$.

Comparing the cardinalities of the sets $\bar\sigma^{-1}(\mathcal L)\cap\lambda(\M_{5,1}^{\cdot k})$ for various $k\in\{1,2,3,4,5\}$ and $\mathcal L\in\lambda(\M_{5,1}^{\cdot 2})$, we see that $\psi(F_2)=F_2$, $\psi(F_3)=F_3$ and $\psi(\A)=\A$ for all $\A\in\M_{5,1}^{\cdot2}\setminus (F_2\cup F_3)$, where  $F_2=\{\triangle_{235},\triangle_{245}\}$ and $F_3=\{\diamondsuit^2_{345},\diamondsuit^3_{245},\diamondsuit^4_{235}\}$.

Let us check that $\psi(\triangle_{235})=\triangle_{235}$. To derive a contradiction, assume that $\psi(\triangle_{235})\ne\triangle_{235}$ and hence $\psi(\triangle_{235})=\triangle_{245}$. Consider the element $\triangle_{124}\in \bar\sigma^{-1}(\triangle_{235})\setminus\lambda(\M_{5,1}^{\cdot2})$ and observe that $\psi(\triangle_{124})\in\psi(\bar\sigma^{-1}(\triangle_{235}))= \bar\sigma^{-1}(\psi(\triangle_{235}))= \bar\sigma^{-1}(\triangle_{245})$.
Let $\pi:\lambda(\M_{5,1})\to\lambda(\M_{5,1})$ be a permutation such that $\pi(\psi(\triangle_{124}))=\triangle_{135}$ and $\pi(\A)=\A$ for all $\A\in\lambda(\M_{5,1})\setminus\{\triangle_{135},\triangle_{124}\}$.
By Corollary~\ref{c:Ker}, the permutation $\pi$ belongs to the automorphism group $\Aut(\lambda(\M_{5,1}))$. Replacing $\psi$ by $\pi\circ\psi$, we can assume that $\psi(\triangle_{124})=\triangle_{135}$.

Observe that $$a^5=\triangle_{135}*\triangle_{123}=
\psi(\triangle_{124})*\psi(\triangle_{123})=\psi(\triangle_{124}*\triangle_{123})=
\psi(\triangle_{345})\ne a^5.$$
This contradiction shows that $\psi(\triangle_{235})=\triangle_{235}$ and hence $\psi(\triangle_{245})=\triangle_{245}$ (as $\psi(F_2)=F_2)$).

Next, we show that $\psi(\A)=\A$ for any $\A\in F_3=\{\diamondsuit^2_{345},\diamondsuit^3_{245},\diamondsuit^4_{235}\}$.
To simplify notations, put $\diamondsuit^2:=\diamondsuit^2_{345}$,
$\diamondsuit^3:=\diamondsuit^3_{245}$, $\diamondsuit^4:=\diamondsuit^4_{235}$.
To derive a contradiction, assume that $\psi(\diamondsuit^i)=\diamondsuit^j$ for some $2\le i<j\le 4$.
Consider the maximal linked upfamily $\Lambda^{i-1}\in \bar\sigma^{-1}(\diamondsuit^i)$ and
observe that $\psi(\Lambda^{i-1})\in\psi(\bar\sigma^{-1}(\diamondsuit^{i}))=
\bar\sigma^{-1}(\psi(\diamondsuit^i))=\bar\sigma^{-1}(\diamondsuit^{j})$.
Since $\Lambda^{j-1}\in \bar\sigma^{-1}(\diamondsuit^{j})$, we can replace $\psi$ by the
composition with the permutation exchanging $\psi(\Lambda^{i-1})$ with $\Lambda^{j-1}$
and $\psi(\triangle_{134})$ with $\triangle_{134}$, and assume that $\psi(\Lambda^{i-1})=\Lambda^{j-1}$ and $\psi(\triangle_{134})=\triangle_{134}$.

 Now observe that the elements $\Lambda^1,\Lambda^2,\Lambda^3$ can be algebraically distinguished by the equalities:
$$
\begin{gathered}
\Lambda^1*\triangle_{123}=\triangle_{345},\;\;\Lambda^1*\triangle_{134}=\triangle_{345}\\
\Lambda^2*\triangle_{123}=\triangle_{345},\;\;\Lambda^2*\triangle_{134}=a^5\\
\Lambda^3*\triangle_{123}=a^5,\;\;\Lambda^3*\triangle_{134}=a^5.
\end{gathered}
$$
Two cases are possible. If $i=2$, then $\psi(\Lambda^1)=\Lambda^{j-1}$ and we obtain a contradiction: $$a^5=\psi(a^5)=\Lambda^{j-1}*\triangle_{134}=\psi(\Lambda^1)*\psi(\triangle_{134})=
\psi(\Lambda^1*\triangle_{134})=\psi(\triangle_{345})\ne a^5.$$
If $i=3$, then we obtain a contradiction considering
$$a^5=\psi(a^5)=\Lambda^{j-1}*\triangle_{123}=\psi(\Lambda^2)*\psi(\triangle_{123})=
\psi(\Lambda^2*\triangle_{123})=\psi(\triangle_{345})\ne a^5.$$
In both cases we obtain a contradiction with the assumption that $\psi(\A)\ne\A$ for some $\A\in F_3$. This contradiction completes the proof of the triviality of the range of the operator $R$. Then the group $\Aut(\lambda(\M_{5,1}))$ is equal to the kernel of the operator $R$ and hence is isomorphic to the group
$
S_4^3\times S_5\times S_9^2\times S_{14}\times S_{18}$.
\newpage

\section{Summary Table and some Conjectures}\label{s:table}

The obtained results on the automorphism groups of superextensions of monogenic semigroups of cardinality $\le 5$ are summed up in the following table.

\begin{center}
\begin{tabular}{|c|c|c|c|c|c|}
\hline
$S$ & $|S|=$& $S\cong$& $\Aut(S)\cong$&$\Aut(\lambda(S))\cong$&$|\Aut(\lambda(S))|=$\\
\hline
$\M_{1,1}$ &1&$C_1$&$C_1$&$C_1$&1\\
\hline
$\M_{1,2}$ &2&$C_2$&$C_1$&$C_1$&1\\
\hline
$\M_{2,1}$ &2&  &$C_1$&$C_1$&1\\
$\M_{2,1}^{\cdot2}$ &1& $C_1$&$C_1$&$C_1$&1\\
\hline
$\M_{1,3}$ &3& $C_3$&$C_2$&$C_2$&2\\
\hline
$\M_{2,2}$ &3& &$C_1$&$C_2$&2\\
$\M_{2,2}^{\cdot2}$ &2& $C_2$&$C_1$&$C_1$&1\\
\hline
$\M_{3,1}$ &3& &$C_1$&$C_1$&1\\
$\M_{3,1}^{\cdot2}$ &2& $\M_{2,1}$&$C_1$&$C_1$&1\\
$\M_{3,1}^{\cdot3}$ &1& $C_1$&$C_1$&$C_1$&1\\
\hline
$\M_{1,4}$ &4& $C_4$&$C_2$&$C_2^2$&4\\
\hline
$\M_{2,3}$ &4& &$C_1$&$S_3\times S_5$&720\\
$\M_{2,3}^{\cdot2}$ &3& $C_3$&$C_2$&$C_2$&2\\
\hline
$\M_{3,2}$ &4& &$C_1$&$S_3\times S_4$&144\\
$\M_{3,2}^{\cdot2}$ &3& $\M_{2,2}$&$C_1$&$C_2$&2\\
$\M_{3,2}^{\cdot3}$ &2& $C_2$&$C_1$&$C_1$&1\\
\hline
$\M_{4,1}$ &4& &$C_1$&$S_3\times S_4$&144\\
$\M_{4,1}^{\cdot2}$ &3&  &$C_2$&$S_3$&6\\
$\M_{4,1}^{\cdot3}$ &2& $\M_{2,1}$&$C_1$&$C_1$&1\\
$\M_{4,1}^{\cdot4}$ &1& $C_1$&$C_1$&$C_1$&1\\
\hline
$\M_{1,5}$ &5& $C_5$&$C_4$&$C_4$&4\\
\hline
$\M_{2,4}$ &5& &$C_1$&$S_3^3{\times} S_8^3{\times} S_{17}{\times} S_{19}$&4\,465\,152\\
$\M_{2,4}^{\cdot2}$ &4& $C_4$&$C_2$&$C_2^2$&4\\
\hline
$\M_{3,3}$ &5& &$C_1$&$S_4^3{\times} S_{5}{\times} S_9^2{\times} S_{14}{\times} S_{18}$&6\,531\,840\\
$\M_{3,3}^{\cdot2}$ &4& $\M_{2,3}$ &$C_1$&$S_3\times S_5$&720\\
$\M_{3,3}^{\cdot3}$ &3&$C_3$&$C_2$&$C_2$&2\\
\hline
$\M_{4,2}$ &5& &$C_1$&$S_4^3{\times} S_5{\times} S_9^2{\times} S_{14}{\times} S_{18}$&6\,531\,840\\
$\M_{4,2}^{\cdot2}$ &4& & $C_2$&$S_5\times S_5$&14400\\
$\M_{4,2}^{\cdot3}$ &3& & $C_1$&$C_2$&2\\
$\M_{4,2}^{\cdot4}$ &2&$C_2$&$C_1$&$C_1$&1\\
\hline
$\M_{5,1}$ &5& &$C_1$&$S_4^3{\times} S_5{\times} S_9^2{\times} S_{14}{\times} S_{18}$&6\,531\,840\\
$\M_{5,1}^{\cdot2}$ &4&  &$C_1$&$S_9$&362\,880\\
$\M_{5,1}^{\cdot3}$ &3& $\M_{4,1}^{\cdot2}$&$C_2$&$S_3$&6\\
$\M_{5,1}^{\cdot4}$ &2& $\M_{2,1}$&$C_1$&$C_1$&1\\
$\M_{5,1}^{\cdot5}$ &1& $C_1$&$C_1$&$C_1$&1\\
\hline
\end{tabular}
\end{center}
\bigskip

Analyzing the entries of this table and the arguments in Section~\ref{s4}, we can make the following conjectures.

\begin{conjecture} For any integer numbers $r,s\ge 3$ and $n,m\ge 1$ with $r+m=s+n$ the automorphism groups $\Aut(\lambda(\M_{r,m}))$ and $\Aut(\lambda(\M_{s,n}))$ are isomorphic.
\end{conjecture}

\begin{conjecture} For any integer $r\ge 2$ and $m\ge 1$ the restriction operator $R:\Aut(\lambda(\M_{r,m}))\to\Aut(\lambda(\M_{r,m}^{\cdot2}))$, $R:\psi\mapsto \psi|\lambda(\M_{r,m}^{\cdot2})$,
has trivial range.
\end{conjecture}

\end{document}